\documentclass[10pt,final,journal,letterpaper]{IEEEtran}

\usepackage[]{graphicx}
\usepackage{amssymb, amsmath}
\usepackage{subfig}
\usepackage{algorithm}
\usepackage{algorithmicx}
\usepackage{algpseudocode}

\usepackage{amssymb}
\usepackage{amsmath,amsfonts,amscd}
\usepackage{mciteplus}
\usepackage{cite}
\usepackage{graphicx,epsfig}
\usepackage{citesort}
\usepackage{url}
\usepackage{footmisc}
\usepackage[alwaysadjust]{paralist}
\usepackage{sublabel}
\usepackage{epstopdf}

\newtheorem{theorem}{Theorem}

\newtheorem{lemma}{Lemma}

\newtheorem{example}{Example}

\makeatletter

\newcommand{\Rmnum}[1]{\expandafter\@slowromancap\romannumeral #1@}

\newcommand{\ie}{\emph{i.e.}, }

\newcommand{\qed}{\hfill \ensuremath{\Box}}

\makeatother

\ifCLASSINFOpdf
\else
\fi
\begin{document}
%
\title{Empirical Likelihood Ratio Test with Distribution Function Constraints}
%
%
%

\author{Yingxi~Liu,~\IEEEmembership{Student member,~IEEE,}
        Ahmed~Tewfik,~\IEEEmembership{Fellow,~IEEE}\\
\thanks{Copyright (c) 2012 IEEE. Personal use of this material is permitted. However, permission to use this material for any other purposes must be obtained from the IEEE by sending a request to pubs-permissions@ieee.org.}
\thanks{Y. Liu and A. Tewfik are with the Department of Electrical and Computer Engineering, University of Texas at Austin, Austin, TX, 78712 USA e-mail: yingxi@utexas.edu, tewfik@austin.utexas.edu.}}

\markboth{Paper Draft}%
{Shell \MakeLowercase{\textit{et al.}}: Bare Demo of IEEEtran.cls for Journals}
%



\maketitle

\begin{abstract}
In this work, we study non-parametric hypothesis testing problem with distribution function constraints. The empirical likelihood ratio test has been widely used in testing problems with moment (in)equality constraints. However, some detection problems cannot be described using moment (in)equalities. We propose a distribution function constraint along with an empirical likelihood ratio test. This detector is applicable to a wide variety of robust parametric/non-parametric detection problems. Since the distribution function constraints provide a more exact description of the null hypothesis, the test outperforms the empirical likelihood ratio test with moment constraints as well as many popular goodness-of-fit tests, such as the robust Kolmogorov-Smirnov test and the Cram\'{e}r-von Mises test. Examples from communication systems with real-world noise samples are provided to show their performance. Specifically, the proposed test significantly outperforms the robust Kolmogorov-Smirnov test and the Cram\'{e}r-von Mises test when the null hypothesis is nested in the alternative hypothesis. The same example is repeated when we assume no noise uncertainty. By doing so, we are able to claim that in our case, it is necessary to include uncertainty in noise distribution. Additionally, the asymptotic optimality of the proposed test is provided.

\end{abstract}


\IEEEpeerreviewmaketitle

\section{Introduction}
This paper proposes a robust hypothesis testing strategy. Likelihood ratio tests are optimal statistical tests for comparing two hypotheses with known statistical descriptions. When the statistical description of one or both hypotheses includes parameters with uncertain values or the data is drawn from a family of probability distributions under one or both hypotheses, one needs to apply robust and/or non-parametric tests. In particular, the test that we propose falls under the category of empirical likelihood ratio tests. A special class of such tests was first studied  by A. Owen \cite{owen1988empirical, owen1990empirical, owen2001empirical} to test the validity of moment equalities.  Robust tests for moment constraints are usually referred to as empirical likelihood ratio tests with moment constraint. The empirical likelihood ratio test with moment constraint (ELRM) is widely used as a tool for non-parametric detection problems in economics. However, the application of the ELRM is largely limited because of the inability of moment constraints to efficiently capture the characteristics of general problems in practice. In particular, the test loses its power when one or both targeted hypotheses cannot be accurately described by moment constraints.

\par
To overcome this shortcoming, our proposed test replaces moment constraints with a constraint on the empirical distribution function (EDF) of the observations. The distribution function constraint is simply a set of EDFs.  Each hypothesis in the problems we consider can be a family cumulative distribution functions (CDFs). A hypothesis can also be described by a bounded region that specifies upper and lower bounds on the EDF covered by the hypothesis.

\subsection{Our contribution}
This work makes three major contributions. Firstly, it proposes a novel non-parametric test ELRDF that handles uncertainties of the distribution function in the hypotheses. Its innovative way of modeling the uncertainty region gives rise to new solutions to a large variety of robust detection problems. It can be extremely useful when a set of EDFs can be observed a priori. For example, in Huber's original robust detection problem \cite{huber1965robust}, the true probability distribution $Q$ is buried in an $\epsilon$-contamination model $Q=(1-\epsilon)P+\epsilon H$, where $P$ is the nominal distribution and $H$ an arbitrary distribution function. The $\epsilon$-contamination model can be treated as a variation of the distribution function constraint.

\par
Secondly, our test improves the performance of robust non-parametric tests. Several tests proposed in the past are applicable to our particular problem, including the robust versions of the Kolmogorov-Smirnov (KS) test \cite{unnikrishnan2010thresholds, massey1951kolmogorov, lilliefors1967kolmogorov} and the Cram\'{e}r-von Mises test \cite{anderson1962distribution}. We study the performance of the ELRDF and the robust KS test and the Cram\'{e}r-von Mises test in several examples in a communication system with real-world non-Gaussian noise data acquired from software-defined radio device. Results show that the ELRDF outperforms the robust KS test and the Cram\'{e}r-von Mises test.

\par
Thirdly, it discusses the asymptotic optimality of the ELRDF in the Hoeffding's sense \cite{hoeffding1965asymptotically}, and provide a proof of that fact following the steps in \cite{zeitouni1991universal, qin1994empirical, canay2010inference, kitamura2012asymptotic}. Also, when there is no uncertainty in the null hypothesis, we show that the ELRDF takes a simple formulation. A sample grouping method is proposed to boost the performance of the test.

\subsection{Related work}
The problem under study in the paper covers a wide range of applications. Several studies in the literature are closely related to this work. As previously discussed, Huber's robust detection problem can be treated as a special case of our problem. In \cite{huber1965robust}, Huber's test features a clipped version of the likelihood ratio test between the nominal densities that delivers performance which minimizes the worst-case probability of false alarm and miss. However, the clipped test has limited application since it requires parametric model for the nominal probability distributions. In \cite{levy2009robust}, the author provides a framework of the robust likelihood ratio test when the uncertainty region is described by the Kullback-Leibler divergence. In the parametric case, this problem is also known as the composite hypothesis testing problem \cite[p.~169]{levy2008principles}. In this case, the generalized likelihood ratio test (GLRT) has optimal error performance when the probability distributions under all the hypotheses are in the same exponential family, c.f. \cite[p.~204]{levy2008principles}, \cite{zeitouni1992generalized}. Again, this test requires complete parametric description of the probability distributions.

\par
Similar problems are also studied in the application of the signal detection in noise with uncertainty \cite{tandra2008snr} or non-Gaussian noise \cite{middleton1979canonical, middleton1999non}. Noise uncertainty considers the case where the test designers do not know the noise statistics perfectly. For example, it might be known that the mean or the variance of the noise falls onto an interval but its exact value may be unknown. Non-Gaussian noise considers the case where the noise statistics is a mixture of Gaussian densities of various means and variances. The distribution function constraint can be applied whenever robustness is needed. Besides applications in signal processing, the proposed test has many practical applications, such as quality assurance in manufacturing, event forecasting, assessment of model fitting in finance, to name just a few.

\par
Another class of close relatives of the ELRDF are the non-parametric goodness-of-fit tests \cite{d1986goodness}. The most popular tests among this class include the previously introduced Kolmogorov-Smirnov (KS) test \cite{massey1951kolmogorov, lilliefors1967kolmogorov}, the Cram\'{e}r-von Mises test \cite{anderson1962distribution}, as well as other tests such as the Anderson-Darling test \cite{anderson1954test}, the Shapiro-Wilk test \cite{shapiro1965analysis} for normally distributed null hypothesis. Specifically, a robust version of the KS test was proposed in \cite{unnikrishnan2010thresholds}, where the distribution function constraint fits perfectly. As a result, this test is the closest competitor to our proposed test. In the test, a worst-case Kolmogorov-Smirnov test statistic is computed and compared to a threshold. Another closely related non-parametric test is the ensemble $\phi$-divergence test \cite{kundargiframework}. This test uses the fact that the $\phi$-divergence computes the difference of two distribution with tunable emphasis on different location of the distribution function \cite{jager2007goodness}, and combine the statistics with different emphases to form a new test. This test is proven to be powerful with non-Gaussian noise. So far, a robust version of it is lacking. Another approach is to use the empirical likelihood ratio test with moment inequalities \cite{canay2010inference, rosen2008confidence}, which is one type of moment constraints. This technique suffers from the problem of insufficient description using moment inequalities. For example, when the unknown hypothesis contains the null hypothesis (nested hypotheses), ELRM performs only slightly better than flipping a fair coin.

\par
The rest of this paper is organized as follows. Section \ref{sec:2} formulates the ELRDF test and discusses its asymptotic optimality. Section \ref{sec:3} presents the formulation of ELRDF when there is no uncertainty. Section \ref{sec:4} discusses several other popular goodness-of-fit tests such as the robust KS test and the Cram\'{e}r-von Mises test, and their formulation in the presence of uncertainty. A real-world noise sample set is studied and two examples from communication systems with those noise samples are provided to show the performance of the ELRDF in Section \ref{sec:5}. Section \ref{sec:6} concludes the paper.

\section{Distribution Function Constrained Detection Problem}\label{sec:2}
\subsection{Problem formulation}
Consider a sequence of observations ${\bf X}^n=\{X_i:i=1, \ldots, n, X_i\in\mathcal{X}\}$ which are independently and identically distributed (i.i.d.) probability density $f$, with cumulative distribution function (CDF) $F$. $\mathcal{X}\subseteq\mathbb{R}$ denotes the sample space. Additionally, the empirical CDF with observations ${\bf X}^n$ is denoted as $F_e$
\begin{align}
    F_e(x, {\bf X}^n)=\frac{1}{n}\sum_{i=1}^n{\bf 1}_{\{X_i\leq x\}},\ x\in\mathbb{R},\nonumber
\end{align}
where ${\bf 1}_{\{\cdot\}}$ is the indicator function. Denote $\mathcal{F}_e=\{F_e(x, {\bf X}^n): {\bf X}^n\in\mathcal{X}^n\}$ as the set of all empirical distribution functions on the $n$-dimensional samples space $\mathcal{X}^n$. In the context where ${\bf X}^n$ is provided, we usually write $F_e(x, {\bf X}^n)$ simply as $F_e(x)$. Given ${\bf X}^n$, the problem of whether $F$ belongs to a certain set of probability densities $\mathcal{F}$ is of interest. This is a universal hypothesis testing problem
\begin{align}\label{eq:basicdetectionprob}
    \mathcal{H}_0:&\ F\in\mathcal{F},\nonumber\\
    \mathcal{H}_1:&\ F\notin\mathcal{F}.
\end{align}
We are particularly interested in the form of $\mathcal{F}$ that is characterized by boundaries of certain CDFs, specifically
\begin{align}\label{eq:boundarycondition}
    \mathcal{F}=\{G:\ F_l(x)\leq G(x)\leq F_u(x)\}.
\end{align}

\subsection{Solution}
Given ${\bf X}^n$, let $\hat{F}$ be a CDF that is absolutely continuous with respect to $F_e$ ($\hat{F}\ll F_e$) and $w_i=\hat{F}(X_i)-\hat{F}(X_i-)$, where $\hat{F}(X_i-)=\lim\limits_{x\rightarrow X_i-}\hat{F}(x)$, which is the value of $\hat{F}(x)$ approaching $X_i$ from the left of the x-axis. Denote $l(F_e)=\prod\limits_{i=1}^n\big(F_e(X_i)$ $-F_e(X_i-)\big)=n^{-n}$ and $l(\hat{F})=\prod\limits_{i=1}^nw_i$. The empirical likelihood ratio is defined as
\begin{align}\label{eq:elr}
    R(\hat{F}, F_e)=\frac{l(\hat{F})}{l(F_e)}.
\end{align}
Naturally, $w_i\geq 0$, $\sum_{i=1}^nw_i=1$. We can rewrite $R(\hat{F}, F_e)$ as a function of $\vec{w}=[w_1, w_2, \ldots, $ $w_n]^T$
\begin{align}
    R(\vec{w}, F_e)=\prod\limits_{i=1}^nnw_i.\nonumber
\end{align}
In the sequel, we use $R(\hat{F}, F_e)$ and $R(\vec{w}, F_e)$ interchangeably depending on the context. It is known that $l(\hat{F})\leq l(F_e)$ for all choices of $\vec{w}$ in the probability simplex \cite[p.~8]{owen2001empirical}. When $w_i=\frac{1}{n}$ for all $i$, $l(\hat{F})=l(F_e)$, then $R(\hat{F}, F_e)\leq 1$. As a first step towards the detection problem, we would like to maximize the empirical likelihood ratio $R(\hat{F}, F_e)$ with respect to $\vec{w}$ when $\hat{F}$ satisfies the boundary conditions
\begin{align}
    \max\limits_{\vec{w}}\Big\{R(\vec{w}, F_e):&w_i\geq 0, \sum_{i=1}^nw_i=1,\nonumber\\
    &F_l(X_i)\leq \hat{F}(X_i)\leq F_u(X_i)\Big\}.\nonumber
\end{align}
We shall assume without loss of generality that $X_1<X_2<\ldots<X_n$. Construct a $(n-1)\times n$ matrix
\begin{align}
    A=\begin{bmatrix}
        1 & 0 & 0 & \ldots & 0 & 0\\
        1 & 1 & 0 & \ldots & 0 & 0\\
        1 & 1 & 1 & \ldots & 0 & 0\\
        \vdots & \ddots & \ddots & \ddots & \vdots & \vdots\\
        1 & \ldots & 1 & 1 & 1 & 0\\
      \end{bmatrix},\nonumber
\end{align}
and let $\vec{F_l}=(F_l(X_1), F_l(X_2), \ldots, F_l(X_{n-1}))^T$, $\vec{F_u}=(F_u(X_1), F_u(X_2), \ldots, F_u(X_{n-1}))^T$. The last constraint is conveniently written as
\begin{align}
    \vec{F_l}\leq A\vec{w}\leq\vec{F_u}.\nonumber
\end{align}
One should notice that $A$ does not contain a row of all ones since the constraint $\sum_{i=1}^nw_i=1$ will certainly contradict the assertion $F_l(X_n)\leq\sum_{i=1}^nw_i\leq F_u(X_n)$. Indeed, one can also drop the constraint $\sum_{i=1}^nw_i=1$ and add an all-one row to the bottom of $A$. We shall see that at this point, it would not make a dramatic difference to favor one alternative over the other. We formally introduce the empirical likelihood with distribution function constraints as follows
\begin{align}\label{eq:eldopt}
    \max\limits_{\vec{w}}\Big\{R(\vec{w}, F_e):&w_i\geq 0, \sum_{i=1}^nw_i=1, \vec{F_l}\leq A\vec{w}\leq\vec{F_u}\Big\}.
\end{align}
This is a problem with a concave objective function (after taking $\log$ operation) and linear constraints. The solution to it is readily available. Let $\vec{w}^{\ast}$ be the maximizer and corresponding CDF as $F^{\ast}$. We build the empirical likelihood ratio test with distribution function constraints on the value of $R(\vec{w}^{\ast}, F_e)$
\begin{align}\label{eq:elrdetector}
    -\frac{1}{n}\log R(\vec{w}^{\ast}, F_e)\overset{\mathcal{H}_1}{\underset{\mathcal{H}_0}{\gtrless}}\eta,
\end{align}
where $\eta\geq 0$. One can immediately identify that $-\frac{1}{n}\log R(\vec{w}^{\ast}, F_e)=D(F_e||\vec{w}^{\ast})$, where $D(\cdot||\cdot)$ is the Kullback-Leibler divergence. The test is to say that when the estimated likelihood is close enough to the empirical distribution, we declare that $\mathcal{H}_0$ is true. Otherwise we declare $\mathcal{H}_1$ true.

\subsection{Asymptotic optimality}
The test \eqref{eq:elrdetector} is essentially a partition of $\mathcal{F}_e$. Denote the partition as $\Lambda(n)=(\Lambda_0(n), \Lambda_1(n))$ where $\mathcal{F}_e=\Lambda_0(n)\cup\Lambda_1(n)$, $\Lambda_1(n)=\Lambda_0^c(n)$, and
\begin{align}
    \Lambda_0(n)=\{F_e:-\frac{1}{n}\log R(F^{\ast}, F_e)\leq\eta\}.\nonumber
\end{align}
We also refer to the partition $\Lambda(n)$ as the test with sample size $n$. We refer to the test simply as $\Lambda$ in the context of asymptotics. Consider an arbitrary test $\Omega(n)=(\Omega_0(n), \Omega_1(n))$ with $\mathcal{F}_e=\Omega_0(n)\cup\Omega_1(n)$, $\Omega_1(n)=\Omega_0^c(n)$. The test declares $\mathcal{H}_0$ true if $F_e\in\Omega_0(n)$. The error performance of the test is characterized by the \emph{worst-case} probability of false alarm and probability of miss
\begin{align}
    P_F&=\sup\limits_{F\in\mathcal{F}}F(F_e\in\Omega_1(n)),\nonumber\\
    P_M&=\sup\limits_{F\notin\mathcal{F}}F(F_e\in\Omega_0(n)).\nonumber
\end{align}
Here, $F(F_e\in\Omega_1(n))$ is the probability that the event $F_e$ belongs to $\Omega_1(n)$ happens when the samples are generated by the distribution $F\in\mathcal{F}$. Hence $F(F_e\in\Omega_1(n))$ is the probability of false alarm. Similarly, $F(F_e\in\Omega_0(n))$ when $F\notin\mathcal{F}$ is the probability of miss. Taking the supremum over all $\mathcal{F}$ or its complementary set yields the worst-case probability of false alarm or miss. In the asymptotic regime, it is customary to study the exponential decay rates of $P_F$ and $P_M$ as the number of samples tends to infinity. Their error exponents are expressed as
\begin{align}
    e_F(\Omega)&=\lim\inf\limits_{n\rightarrow\infty}-\frac{1}{n}\log\sup\limits_{F\in\mathcal{F}}F(F_e\in\Omega_1(n))\nonumber\\
    &=\lim\inf\limits_{n\rightarrow\infty}\inf\limits_{F\in\mathcal{F}}-\frac{1}{n}\log F(F_e\in\Omega_1(n)),\nonumber
\end{align}
and
\begin{align}
    e_M(\Omega)&=\lim\inf\limits_{n\rightarrow\infty}-\frac{1}{n}\log\sup\limits_{F\notin\mathcal{F}}F(F_e\in\Omega_0(n))\nonumber\\
    &=\lim\inf\limits_{n\rightarrow\infty}\inf\limits_{F\notin\mathcal{F}}-\frac{1}{n}\log F(F_e\in\Omega_0(n)).\nonumber
\end{align}
To characterize the asymptotic properties of the error exponents of $P_F$ and $P_M$, we need to define a special partition of $\Omega(n)$ as follows
\begin{align}\label{eq:smoothing}
    \Omega_1^{\delta}(n)=\bigcup\limits_{\mu\in\Omega_1(n)}B(\mu, \delta),
\end{align}
and
\begin{align}
    \Omega_0^{\delta}(n)=\mathcal{F}_e\setminus\Omega_1^{\delta}(n),\nonumber
\end{align}
where $B(\mu, \delta)$ denotes an open ball centered at $\mu$ with radius $\delta$ equipped with the Levy metric. For convenience, let $\Omega^{\delta}(n)=(\Omega_0^{\delta}(n), \Omega_1^{\delta}(n))$. The test $\Lambda$ is optimal in the Hoeffding's sense \cite{hoeffding1965asymptotically, zeitouni1991universal}. This result is summarized in the following theorem.
\begin{theorem}\label{thm1}
Consider the test $\Lambda$ such that
\begin{align}
    \Lambda_1=\{F_e:-\frac{1}{n}\log R(F^{\ast}, F_e)>\eta\}.\nonumber
\end{align}
Then 1) and 2) are true:
\begin{enumerate}
  \item $e_F(\Lambda)\geq\eta$.
  \item If an alternative test $\Omega$ satisfies
  \begin{align}
    \lim\inf\limits_{n\rightarrow\infty}\inf\limits_{F\in\mathcal{F}}-\frac{1}{n}\log F(F_e\in\Omega_1^{\delta}(n))>\eta\nonumber
  \end{align}
  for some $\delta>0$, then
  \begin{align}
    &\lim\inf\limits_{n\rightarrow\infty}\inf\limits_{F\notin\mathcal{F}}-\frac{1}{n}\log F(F_e\in\Omega_0(n))\nonumber\\
    &\leq\lim\inf\limits_{n\rightarrow\infty}\inf\limits_{F\notin\mathcal{F}}-\frac{1}{n}\log F(F_e\in\Lambda_0).\nonumber
  \end{align}
\end{enumerate}
\end{theorem}
One should first notice that the test is asymptotically consistent: when the null hypothesis is true, $P\{F_e\in\Lambda_0\}\overset{n\rightarrow\infty}\longrightarrow 1$. This is true according to Glivenko-Cantelli theorem \cite{dudley1999uniform, lehmann2005testing}: the empirical distribution uniformly converges to the true distribution. We provide the proof of Theorem \ref{thm1} in Appendix \ref{app1}.

\section{Degenerate distribution function constraint}\label{sec:3}
When $\mathcal{F}=\{F\}$ is simple, we say that the distribution function constraint is degenerate. In this case, it is necessary to replace $\sum_{i=1}^nw_i=1$ by $\sum_{i=1}^nw_i=F(X_n)$ in problem \eqref{eq:eldopt}. With the degenerate constraint, the values of $w_i$'s are fixed
\begin{align}
    w_1&=F(X_1),\nonumber\\
    w_i&=F(X_i)-F(X_{i-1}),i=2,3,\ldots,n.\nonumber
\end{align}
Here we assume without loss of generality that $X_1<X_2<\ldots<X_n$. According to probability integral transformation theorem \cite[p.~108]{lehmann2005testing}, $F(X)$ is uniformly distributed over $[0,1]$ when $X$ is drawn from $F$. Denote $W_i$'s as the random variables associated to the $w_i$'s. We have the following result regarding the statistics of $W_i$'s.
\begin{lemma}\label{lm1}
$W_i$'s are identically distributed $\forall i$ with distribution function
\begin{align}
    f_i(w_i)=n(1-w_i)^{n-1},\nonumber
\end{align}
and $W_i\overset{d.}{\longrightarrow}\frac{1}{n+1}$.
\end{lemma}
\begin{proof}
See Appendix \ref{app2}.
\end{proof}
With this result, the test statistics
\begin{align}
    -\frac{1}{n}\log R(\vec{w}^{\ast}, F_e)\overset{d.}{\longrightarrow}\log(1+\frac{1}{n}),\nonumber
\end{align}
when the null hypothesis is true. This gives us information about how to design the test statistic when the distribution constraint is degenerate. Indeed, a test can be built as follows. Divide the $n$ samples into $k$ small groups of $m$ samples each. Ensure that the samples are randomly selected. The samples are relabeled to as $X_{ij}$, $i=1, 2, \ldots, k$, $j=1, 2, \ldots, m$. In group $i$, order the samples such that $X_{i1}<X_{i2}<\ldots<X_{im}$. Then compute the likelihoods
\begin{align}
    w_{i1}&=F(X_{i1}),\nonumber\\
    w_{ij}&=F(X_{ij})-F(X_{i,j-1}),j=2,3,\ldots,n.\nonumber
\end{align}
Next, average them over $k$ groups
\begin{align}
    \tilde{w}_j=\frac{1}{k}\sum_{i=1}^kw_{ij}.\nonumber
\end{align}
Then $-\frac{1}{m}\log R(\vec{\tilde{w}}^{\ast}, F_e)\overset{k\rightarrow\infty}{\longrightarrow}\log(1+\frac{1}{m})$. In this setting, the test statistic converges much faster than without grouping.

\section{Other Robust Goodness-of-fit Tests}\label{sec:4}
The problem \eqref{eq:basicdetectionprob} is of wide interest as a parametric and non-parametric detection problem. Different solutions have been proposed in the past to address several variations of the problem. For example, in the parametric case, the generalized likelihood ratio test (GLRT) is optimal when the unknown alternative hypothesis is in the same exponential family of the null hypothesis \cite{zeitouni1992generalized}. In the non-parametric case, the empirical likelihood ratio test with moment constraints (ELRM) was proposed to test moment conditions. Some other goodness-of-fit tests might also be applicable, such as the robust KS test and Cram\'{e}r-von Mises test.

The empirical likelihood ratio test with moment constraints is closely related to our proposed detector in that the ELRM also aims to maximize the empirical likelihood ratio. The main difference is that the maximization in ELRM is taken with constraints on the moments rather than distribution functions. Let $g$ be a moment function on the random variable $X_i$'s. The associated moment is $\sum_{i=1}^nw_ig(X_i)$. In ELRM, testing whether the moment of the null hypothesis falls into certain region is of interest. For example, one can consider testing the null hypothesis that the moment is bounded in a scalar interval $l\leq\sum_{i=1}^nw_ig(X_i)\leq u$. Similarly, the following optimization problem is considered
\begin{align}\label{eq:elmopt}
    \max\limits_{\vec{w}}\Big\{\prod\limits_{i=1}^nnw_i:&w_i\geq 0, \sum_{i=1}^nw_i=1,\nonumber\\
    &u\leq \sum_{i=1}^nw_ig(X_i)\leq l\Big\}.
\end{align}
The difference between \eqref{eq:eldopt} and \eqref{eq:elmopt} is in the constraint applied. Because of its constraints, the ELRM has its unique application in financial engineering and economics. For problems to test the validity of moment constraints as specified in \eqref{eq:elmopt}, the detector also enjoys the asymptotic optimality in the Hoeffding's sense.

\par
In general however, the ELRM is unable to succinctly capture the uncertainty in the underlying probability cumulative function. Arbitrary cumulative distribution functions can be completely described only with an infinite number of moments. An infinite or very large number of moment constraints may be needed to capture the region constraint defined in \eqref{eq:boundarycondition}. In such general settings, the proposed method will be more practical than the ELRM.


\subsection{Robust Kolmogorov-Smirnov test}
The Kolmogorov-Smirnov test is a popular non-parametric test. Its robust version is directly applicable to our problem with distribution function constraint $\mathcal{F}$ in \eqref{eq:boundarycondition}. The test statistic of the robust KS test has the following form
\begin{align}\label{eq:robksstat}
    D_n=\inf\limits_{F\in\mathcal{F}}\sup\limits_{x}|F_e(x)-F(x)|.
\end{align}
The test compares the statistic with a constant
\begin{align}\label{eq:ksdetector}
    \sqrt{n}D_n\overset{\mathcal{H}_1}{\underset{\mathcal{H}_0}{\gtrless}}\gamma.
\end{align}
When $\mathcal{F}=\{F\}$ is simple, $\sqrt{n}D_n\overset{d}{\longrightarrow}\sup\limits_{t\in[0,1]}|B(t)|$, where $B(t)$ is the Brownian bridge \cite[p.~585]{lehmann2005testing}. Denote the probability of false alarm as $P_F^{\text{KS}}(\gamma)=Pr\Big \{\sup\limits_{t\in[0,1]}|B(t)|\geq\gamma\Big \}$ for simple $\mathcal{F}$. When $\mathcal{F}$ has the form in \eqref{eq:boundarycondition}, it is verified in \cite{unnikrishnan2010thresholds} that the probability of false alarm of the robust version of KS test $P_F^{\text{RKS}}$ satisfies
\begin{align}
    P_F^{\text{KS}}(\gamma)<P_F^{\text{RKS}}(\gamma)<2P_F^{\text{KS}}(\gamma).\nonumber
\end{align}

\subsection{Cram\'{e}r-von Mises criterion}
Since the KS test measures the ``distance'' between the ECDF and the hypothesized distribution, it belongs to a class of ECDF statistic. Another class of ECDF statistic is the Cram\'{e}r-von Mises family of statistics
\begin{align}
    w^2=\int_{-\infty}^{+\infty}[F_e(x)-F(x)]^2\psi(x)dF(x).\nonumber
\end{align}
Taking $\psi(x)=1$ yields the Cram\'{e}r-von Mises test while taking $\psi(x)=\{F(x)[1-F(x)]\}^{-1}$ yields the Anderson-Darling test. For a discussion of the difference of these two test readers are referred to \cite{kundargiframework, jager2007goodness} for an in-depth discussion. The Cram\'{e}r-von Mises statistic can be further simplified as
\begin{align}
    T_n=nw^2=\frac{1}{12n}+\sum_{i=1}^n\Big[\frac{2i-1}{2n}-F(X_i)\Big]^2.\nonumber
\end{align}
This test rejects the null hypothesis for large value of $T_n$. Considering that $F\in\mathcal{F}$, the robust version of this statistic is
\begin{align}\label{eq:cvmstat}
    T_n^{\text{rob}}=\inf\limits_{F\in\mathcal{F}}\Big\{\frac{1}{12n}+\sum_{i=1}^n\Big[\frac{2i-1}{2n}-F(X_i)\Big]^2\Big\}.
\end{align}
It is obtained by solving the following optimization problem
\begin{align}
    \min\limits_{\vec{F}}\Big\{\sum_{i=1}^n\Big[\frac{2i-1}{2n}-F_i\Big]^2:\vec{F_l}\leq \vec{F}\leq\vec{F_u}, B\vec{F}\geq 0\Big\},\nonumber
\end{align}
where $\vec{F_l}=[F_l(X_1), F_l(X_2),\ldots,F_l(X_n)]^T$, $\vec{F_u}=[F_u(X_1), F_u(X_2),\ldots,F_u(X_n)]^T$, $\vec{F}=[F_1, F_2,$ $\ldots, F_n]^T$, and $B$ is a $(n-1)\times n$ matrix
\begin{align}
    B=\begin{bmatrix}
        -1 & 1 & 0 & \ldots & 0 \\
        0 & -1 & 1 & \ldots & 0 \\
        \vdots & \ddots & \ddots & \ddots & \vdots \\
        0 & \ldots & 0 & -1 & 1 \\
      \end{bmatrix}.\nonumber
\end{align}
The Anderson-Darling test can also be simplified. However, it turns out that its statistic cannot accommodate the robustness requirement. The comparison of the ELRDF, robust KS test and robust Cram\'{e}r-von Mises test will be discussed with examples in the next section.

\section{Examples}\label{sec:5}
In this section, we consider several examples from communication systems where the noise distribution is not perfectly known. In fact, this happens quite frequently in many applications. Firstly, in practice, the noise parameters cannot be known with good precision. Secondly, it is highly possible that other signal sources in the environment might contribute to the noise component. Here, we study the examples with real-world noise samples acquired from a software-defined radio device operating on the 2.49 GHz frequency band. Due to many effects, such as the environment interference, and the imperfections of the hardware, the noise is not perfectly Gaussian nor stationary. To examine whether the noise follows a stationary Gaussian distribution, we consider the KS test for normality \cite[p.~589]{lehmann2005testing}
\begin{align}
    D_n^{\text{normality}}=\sup\limits_{x}|F_e(x)-\Phi(\frac{x-\hat{m}_n}{\hat{\sigma}_n})|,\nonumber
\end{align}
where $(\hat{m}_n, \hat{\sigma}_n)$ is the maximum likelihood estimator for the Gaussian mean and variance with ${\bf X}^n$, and $\Phi(\cdot)$ is the CDF of standard normal. The KS normality test implies that if the noise distribution is Gaussian, its corresponding measure $D_n^{\text{normality}}$ viewed as a random variable should have the same CDF as that of the simulated Gaussian. This holds for any sample size and any Gaussian distribution with any mean and variance. In this study, we use 5 millions real-world noise samples. The CDFs of $D_n^{\text{normality}}$ are plotted in Figure \ref{fig:normalitytest} where we compare the statistics generated by simulated Gaussian noise and the real-world noise samples. In Figure \ref{fig:normality10}, it is observed that the two CDFs agree perfectly with each other. However, when the sample size increases, the two begin to diverge as shown in Figure \ref{fig:normality50} to \ref{fig:normality500}. The $D_n^{\text{normality}}$ measure with real-world noise stochastically dominates the one with simulated Gaussian noise for large sample size. This indicates that the KS normality test has a good chance to separate the real-world noise from Gaussian noise. There is one way to interpret this phenomenon. It is possible that small numbers of consecutively collected noise samples do follow the same Gaussian distribution. But this distribution changes over time. When larger numbers of samples are examined, they follow a mixture of Gaussian distribution rather than a Gaussian distribution which leads to the failure of the normality test for large sample sizes. Due to this fact, it is not recommended to model the sample we study as Gaussian or assume any stationarity for it.
\begin{figure*}[!t]
  \centering
  \subfloat[sample size 10]{\label{fig:normality10}\includegraphics[width=3.0in]{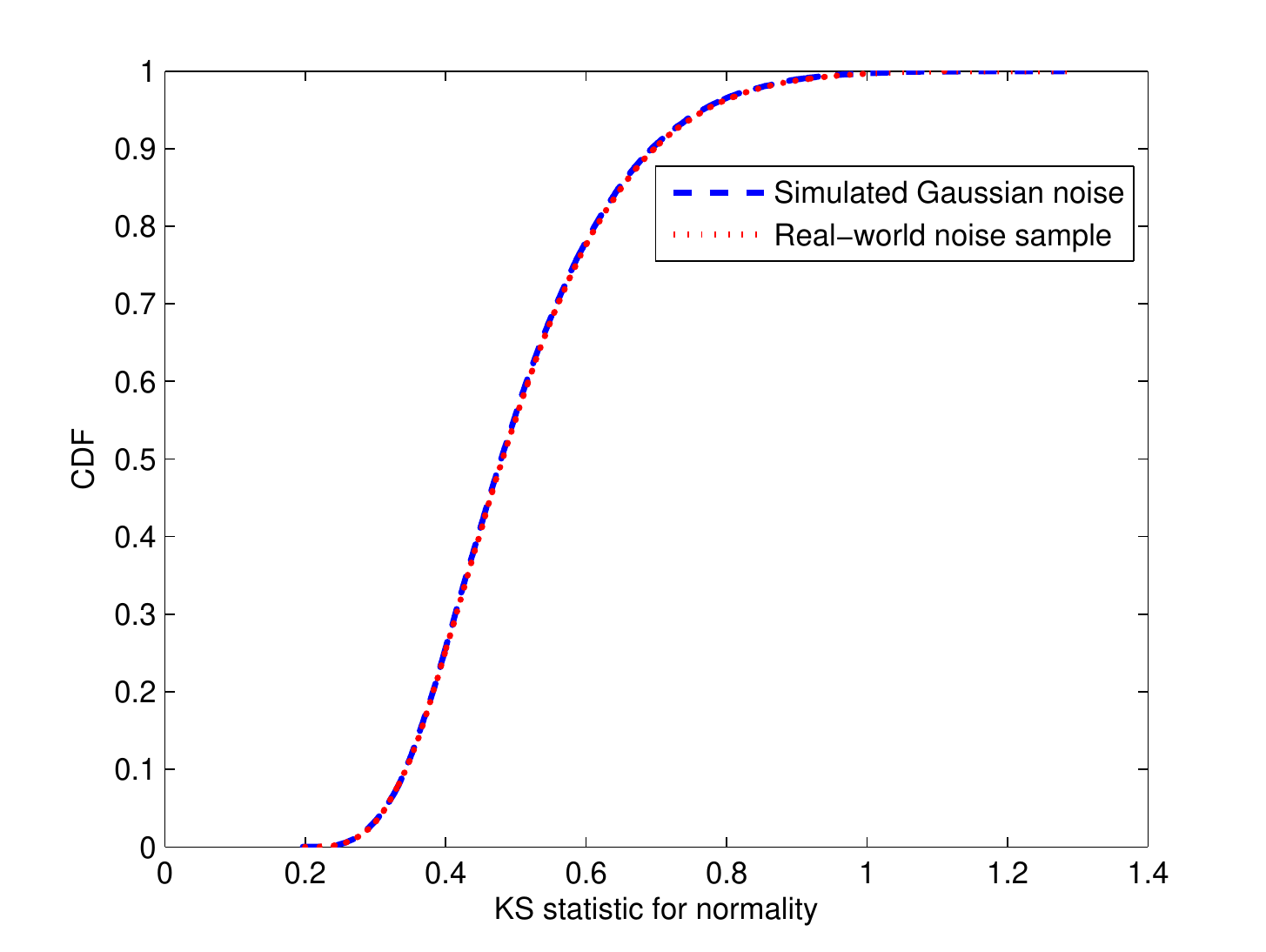}}
  \subfloat[sample size 50]{\label{fig:normality50}\includegraphics[width=3.0in]{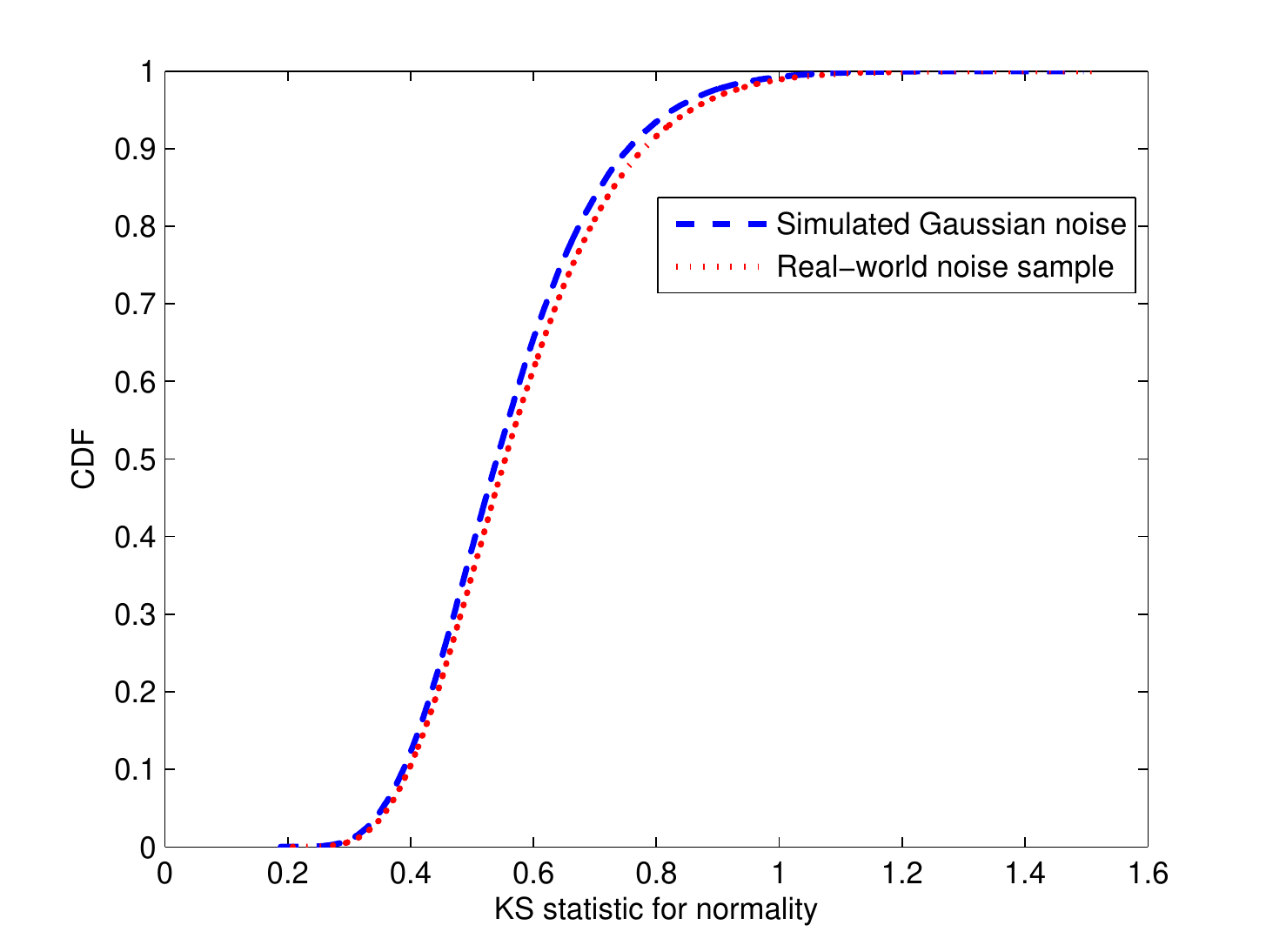}}\\
  \subfloat[sample size 100]{\label{fig:normality100}\includegraphics[width=3.0in]{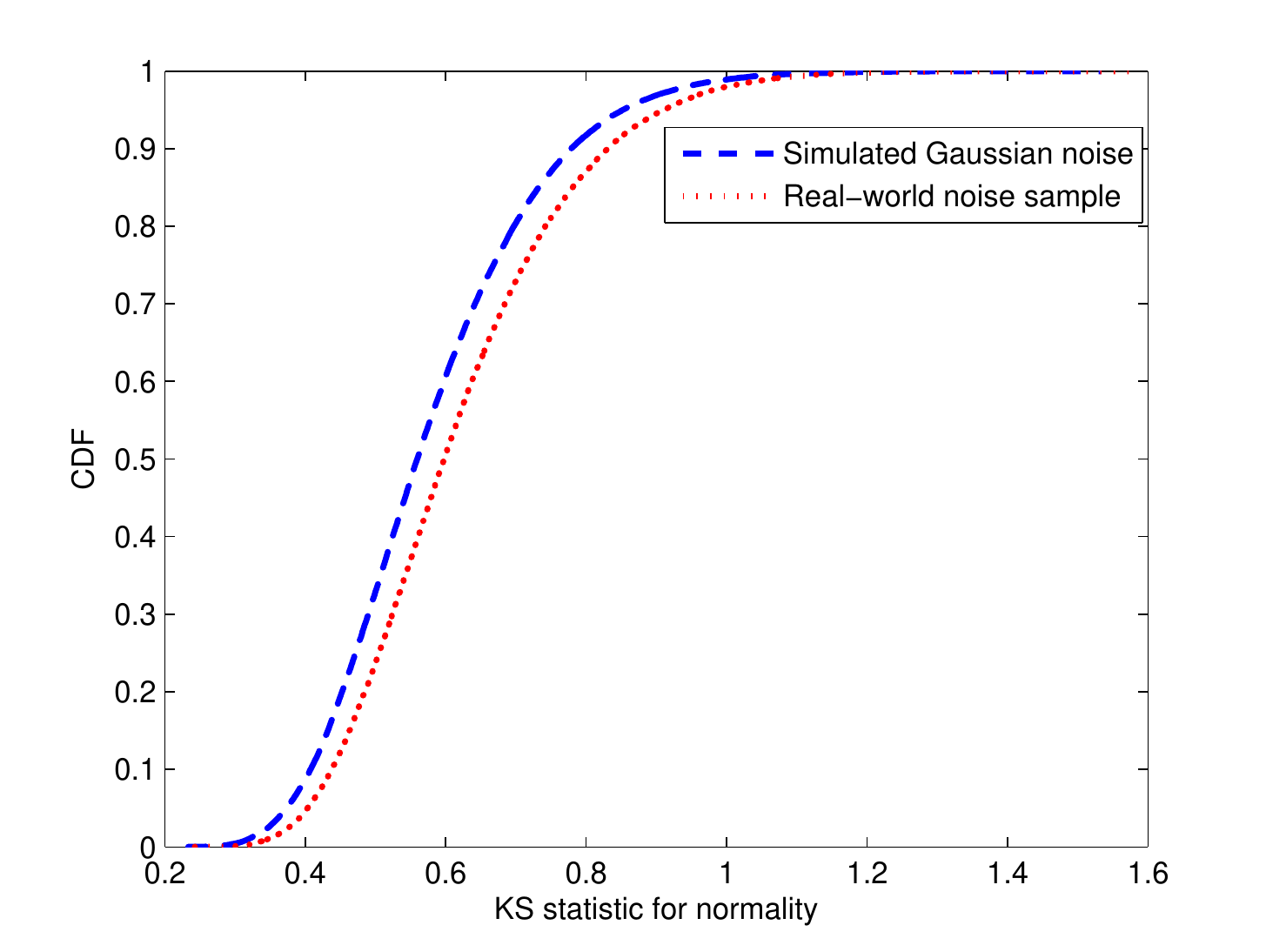}}
  \subfloat[sample size 500]{\label{fig:normality500}\includegraphics[width=3.0in]{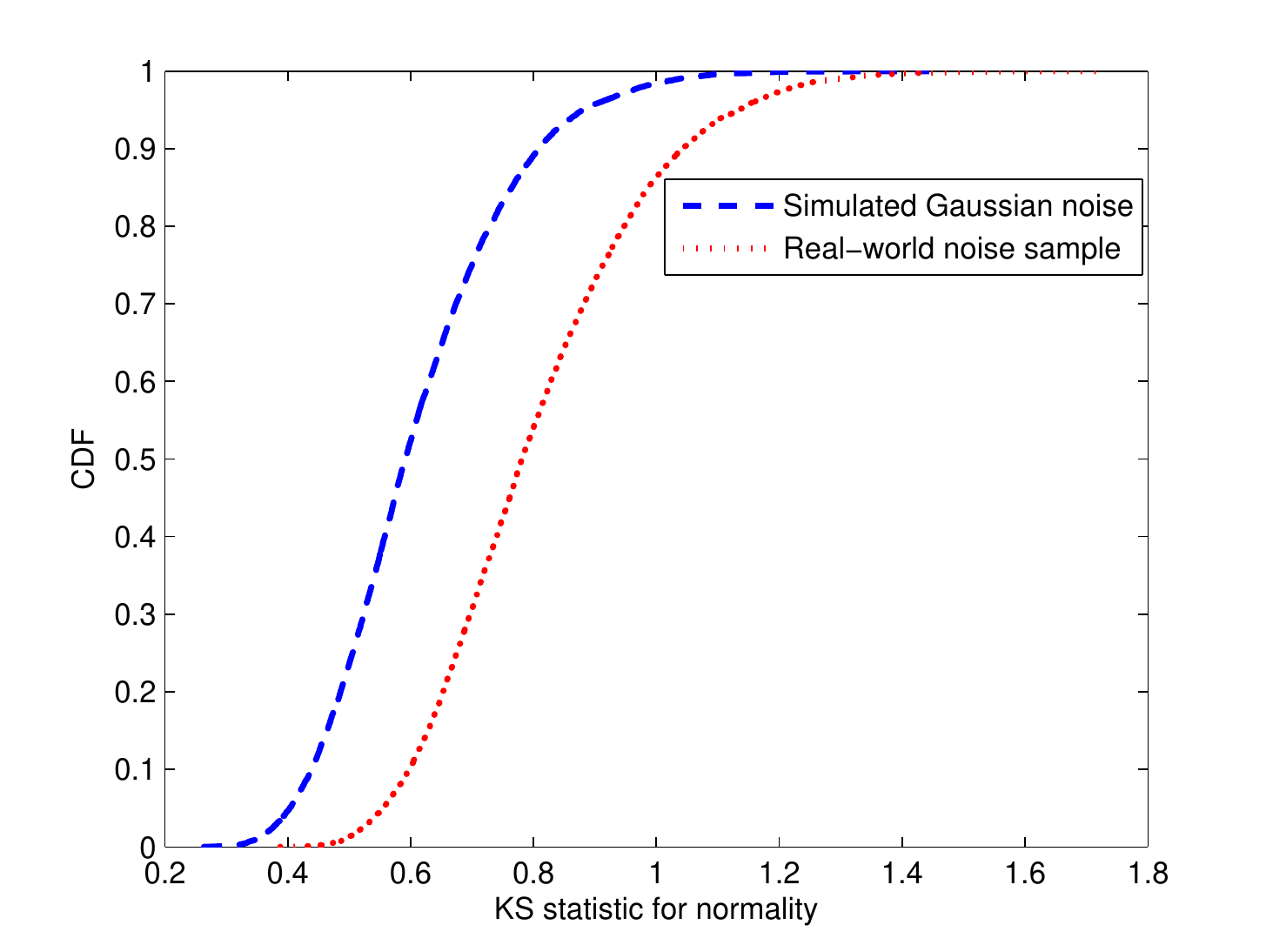}}
  \caption{CDFs of the KS statistics for normality test generated by simulated Gaussian noise and real-world noise samples.}
  \label{fig:normalitytest}
\end{figure*}

\par
We therefore describe the noise distribution in a non-parametric form. We first need a description of the boundary condition \eqref{eq:boundarycondition}. The 5 million samples are divided into small groups of equal number of samples. The empirical cumulative distribution function (ECDF) is computed for each sample group. Then the upper and lower bounds of these ECDFs are extracted as $F_u$ and $F_l$. We consider the group size of 100. The empirical uncertainty region of noise sample distribution is plotted in Figure \ref{fig:ulbound}. Notice that there is no particular reason to set the group size to 100. Indeed, with a larger group size, one can obtain a narrower uncertainty region. Regarding the detector performance, this results in larger probability of detection, but also larger probability of false alarm. Indeed, it will be shown that in general uncertainty region is necessary with the samples we study. This uncertainty region is used in all the upcoming examples.
\begin{figure}[!t]
    \centering
    \includegraphics[width=3.5in]{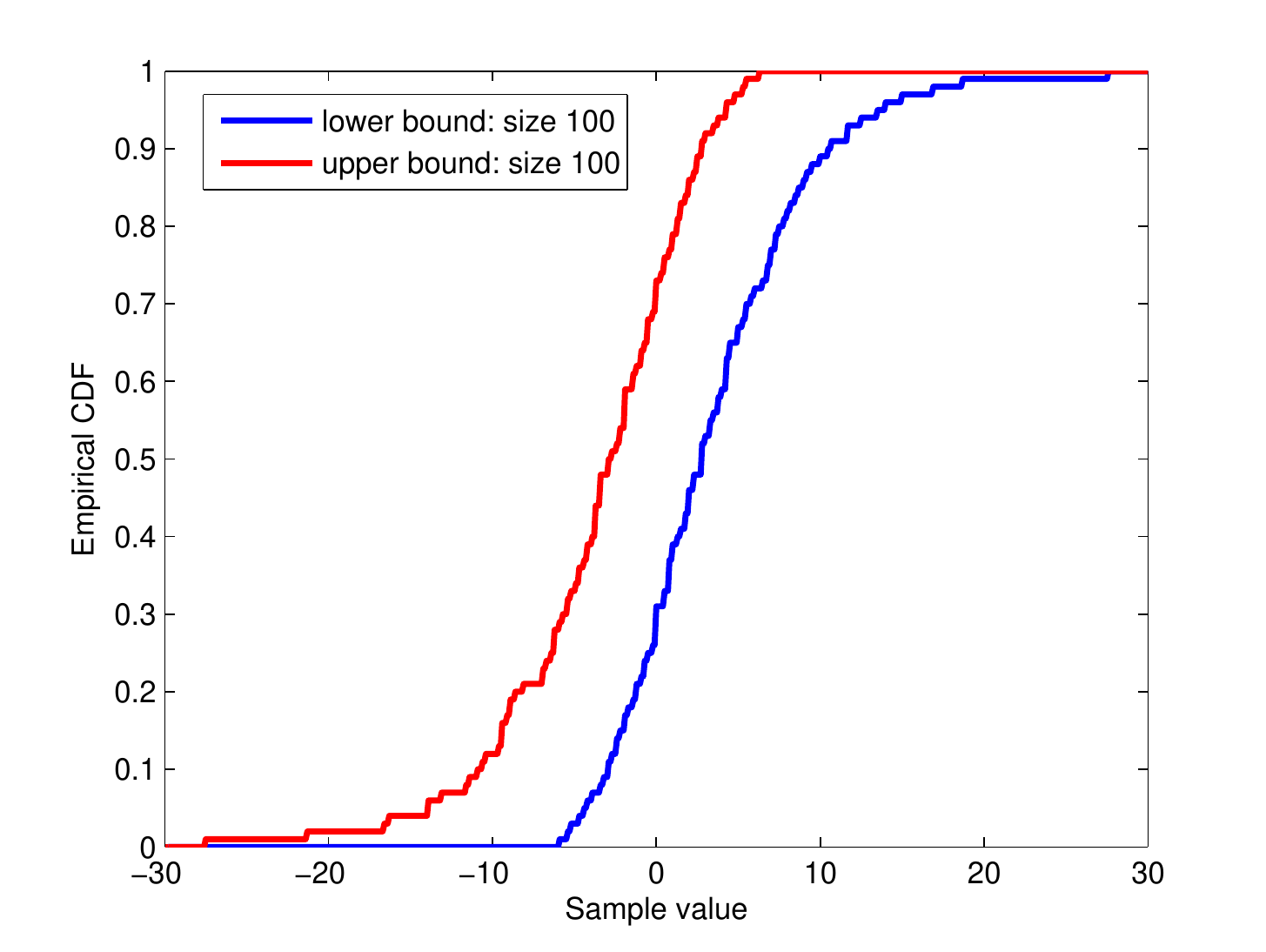}
    \caption{Uncertainty region of experimental noise samples.}
    \label{fig:ulbound}
\end{figure}

\begin{example}[Rich distribution function constraint]
Consider a communication system in which $n$ copies of a binary signal $X$ pass through a channel sequentially at times $i=1,2,\ldots,n$, with gain $h$ and are received with additive noise $v_i$
\begin{align}
    Y_i=h_iX+v_i, i=1,2,\ldots,n.\nonumber
\end{align}
Firstly, we consider the case that the channel is in slow fading, which means that the channel gain $h_i$ is a constant but unknown during the $n$ transmissions. The noise distribution belongs to the region specified in Figure \ref{fig:ulbound}. Without a parametric model of the noise distribution and with no information of the alternative hypothesis, one usually resorts to the goodness-of-fit tests. The performance of ELRDF is compared with the robust Cram\'{e}r-von Mises test and robust Kolmogorov-Smirnov test. For $h_i=3$, their ROC curves with sample size 10 are plotted in Figure \ref{fig:slowfadingp}. It is observed that the robust Cram\'{e}r-von Mises test outperforms the other two tests, including ELRDF. The ELRDF performs slightly better than the robust KS test. But with $h_i=-3$, both the ELRDF and robust KS test significantly outperform the Cram\'{e}r-von Mises test. The reason that the performances differ for $h_i=3$ and $-3$ can be explained as follows. In ELRDF, the objective is being maximized. As a result, the maximizing CDF of ELRDF is being ``pulled'' to the left as much as possible, while the Cram\'{e}r-von Mises test tries to stay at the center, as shown in Figure \ref{fig:singlerun}. When the alternative hypothesis is true, the maximizing CDF of ELRDF is further away than that of the Cram\'{e}r-von Mises test. Due to this fact, the Cram\'{e}r-von Mises test performs better. When the alternative hypothesis is to the left of the null hypothesis, the opposite happens, \ie the ELRDF performs better.

\par
In the fast fading scenario, the null and alternative hypotheses are nested. In the $n$ transmissions, the channel gain $h_i$, $i=1,2,\ldots,n$ is i.i.d. with uniform distribution in $[-10, 10]$. Their ROC curves are shown in Figure \ref{fig:fastfading}. In this case, ELRDF outperforms the other tests.
\begin{figure*}[!t]
  \centering
  \subfloat[$h_i=3$]{\label{fig:slowfadingp}\includegraphics[width=3.5in]{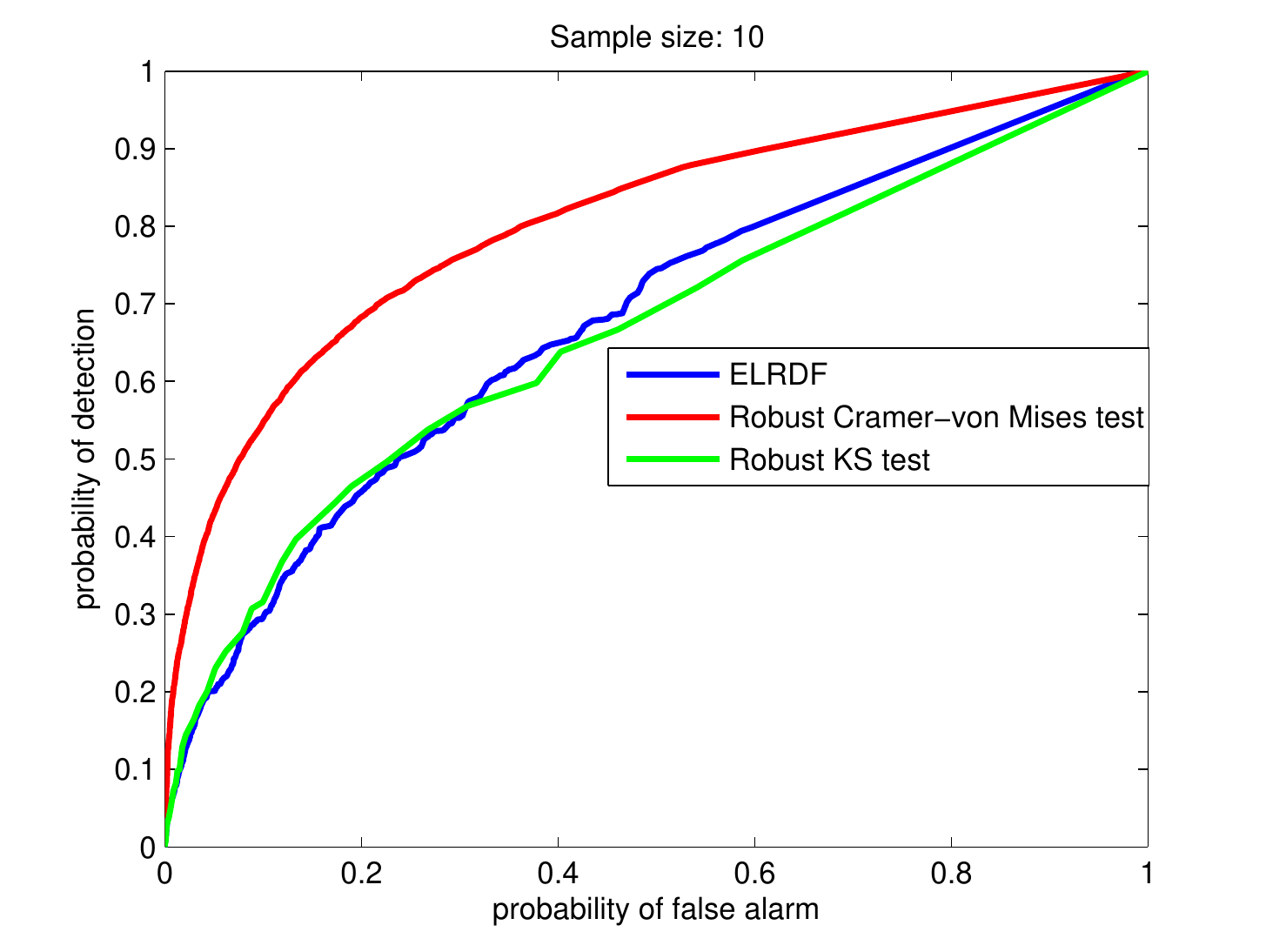}}
  \subfloat[$h_i=-3$]{\label{fig:slowfadingn}\includegraphics[width=3.5in]{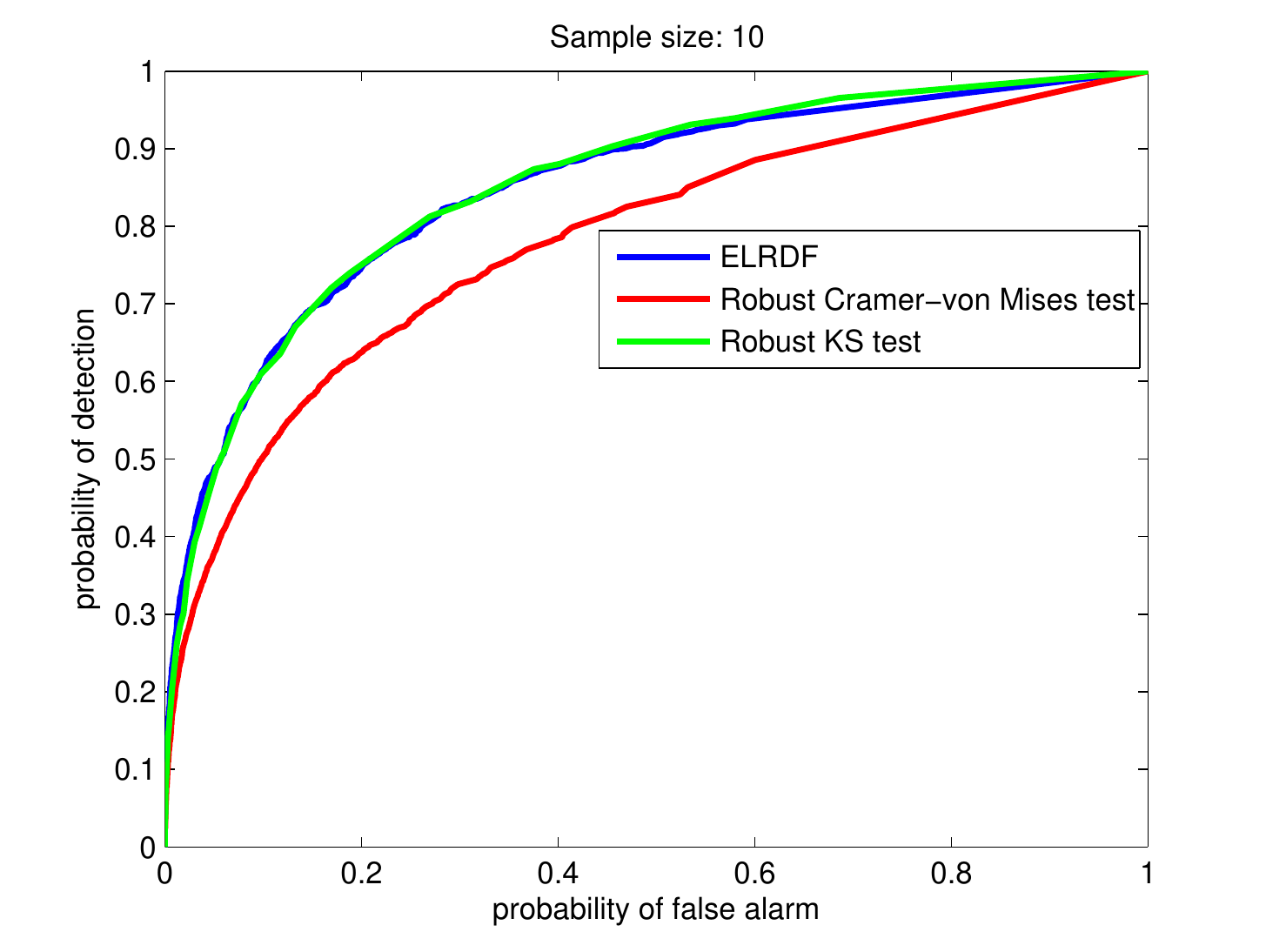}}
  \caption{ROC curves in slow fading scenario with constant channel gain $h_i=3,-3$ with distribution function constraint $\mathcal{F}$ specified in Figure \ref{fig:ulbound}.}
  \label{fig:slowfading}
\end{figure*}
\begin{figure}[!t]
  \centering
  \includegraphics[width=3.5in]{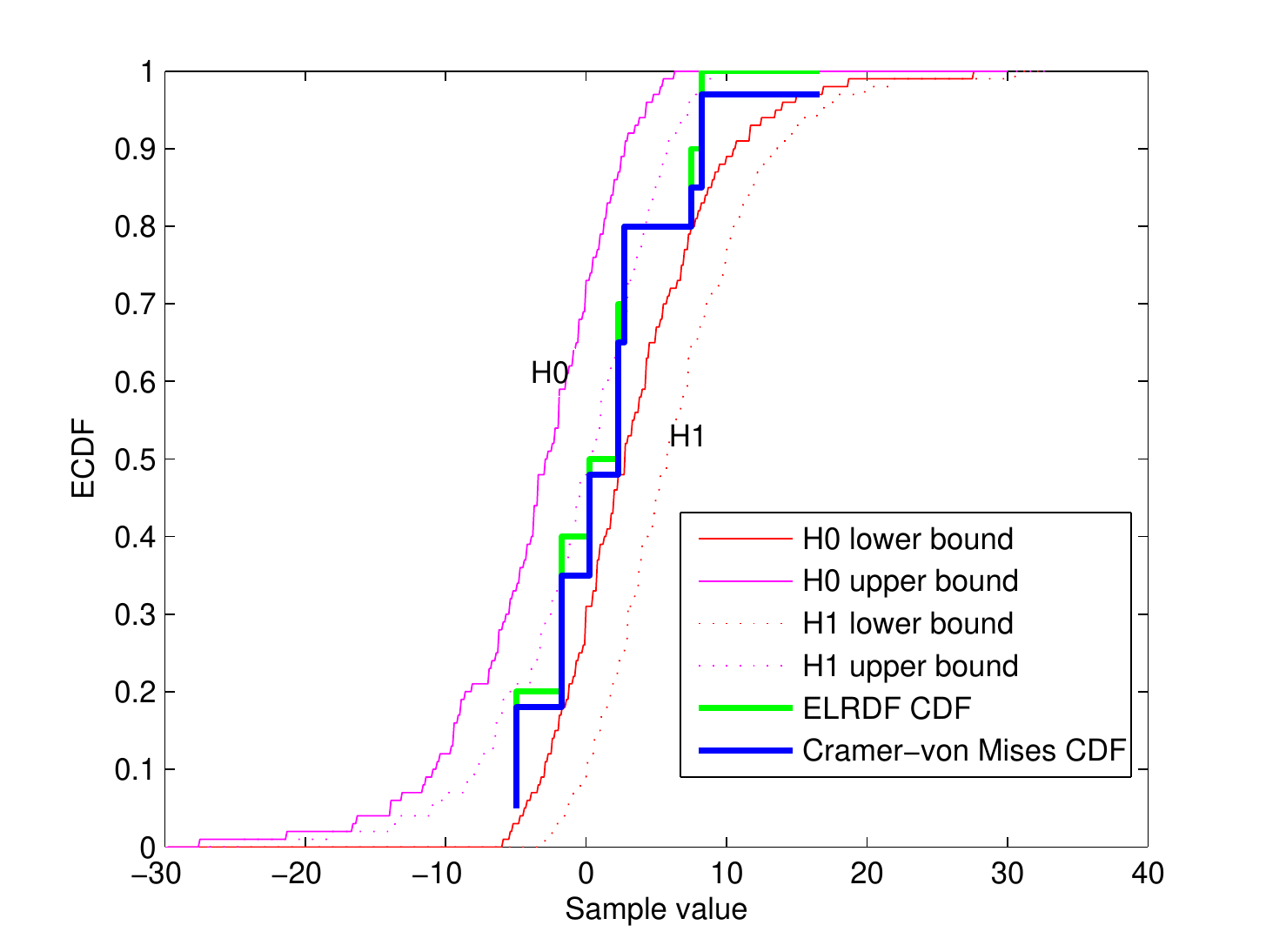}
  \caption{The ECDFs of a single run of the ELRDF and Cram\'{e}r-von Mises test when $h_i=3$ when the alternative hypothesis $\mathcal{H}_1$ is true. In this case, the true distribution is to the right side of the null hypothesis. The Cram\'{e}r-von Mises test is more close to the alternative than the ELRDF.}
  \label{fig:singlerun}
\end{figure}
\begin{figure}[!t]
  \centering
  \includegraphics[width=3.5in]{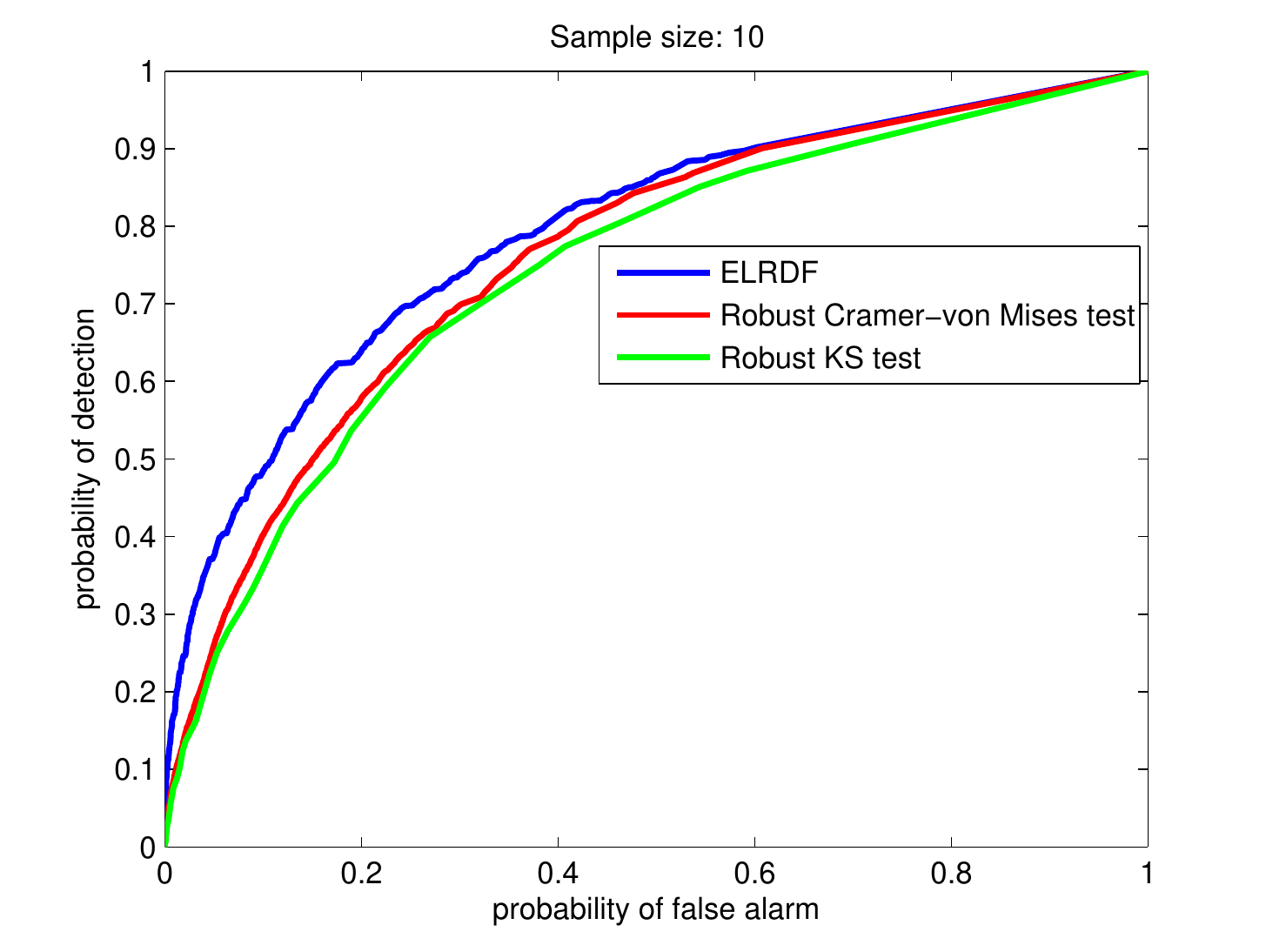}
  \caption{ROC curves in fast fading scenario with i.i.d. channel gain uniformly distributed in $[-10, 10]$ and distribution function constraint $\mathcal{F}$ specified in Figure \ref{fig:ulbound}.}
  \label{fig:fastfading}
\end{figure}
\qed
\end{example}
\begin{example}[Degenerate distribution function constraint]
We consider the same communication system in slow and fast fading environment in the previous example except that here a degenerate distribution function constraint is applied. In this case, the noise follows a non-parametric distribution which is simply the ECDF evaluated with the 5 million samples. With the degenerate constraint, the computation of ELRDF and robust Cram\'{e}r-von Mises test is much simpler. The ELRDF with sample grouping is also considered. The ROC curves in slow fading scenario with $h_i=3$ for all $i$ are presented in Figure \ref{fig:eg3p} with sample size 10, small group size of 1, 2, 5 and 10 for ELRDF. In the slow fading case, it is observed that ELRDF averaging over many groups with small number of samples gains the advantage over averaging with less number of groups but with many samples in each group. Specifically, ELRDF with averaging over 10 groups with 1 sample each outperforms all other tests, including the Cram\'{e}r-von Mises test and the KS test, while ELRDF with averaging over 2 groups with 5 samples each, and without averaging completely fail. Due to previous experiences, we also examine the performances with $h_i=-3$. In this case, the performances of the ELRDF with averaging over small number of large groups also get improved.

\par
In the fast fading scenario, a quite opposite result is observed as shown in Figure\ref{fig:eg4}. Similarly, the channel gain $h_i$, $i=1,2,\ldots,n$ is i.i.d. with uniform distribution in $[-10, 10]$. The ELRDF without averaging outperforms all other tests, including the Cram\'{e}r-von Mises test and the KS test. The ELRDF with averaging over 5 and 10 groups completely fail. With the notations used in Section \ref{sec:3}, considering averaging over 10 size 1 groups, the inability of ELRDF with averaging over many small groups can be explained as follows. When the alternative hypothesis is true, sample value $X_{i1}$ can be very large or very small. As a result, $w_{i1}$ is either close to 0 or 1. But their average $\tilde{w}_1$ is somewhere in the middle of $[0,1]$. When the null hypothesis is true, $\tilde{w}_1$ is also somewhere in the middle of $[0,1]$. Therefore it becomes difficult to separate the two hypotheses. We also notice that some of the curves fall in the lower right triangle, which is not permitted by the definition of ROC curves. To make them perform correctly, one needs to switch the roles of probability of detection and false alarm.
\begin{figure*}[!t]
  \centering
  \subfloat[$h_i=3$]{\label{fig:eg3p}\includegraphics[width=3.5in]{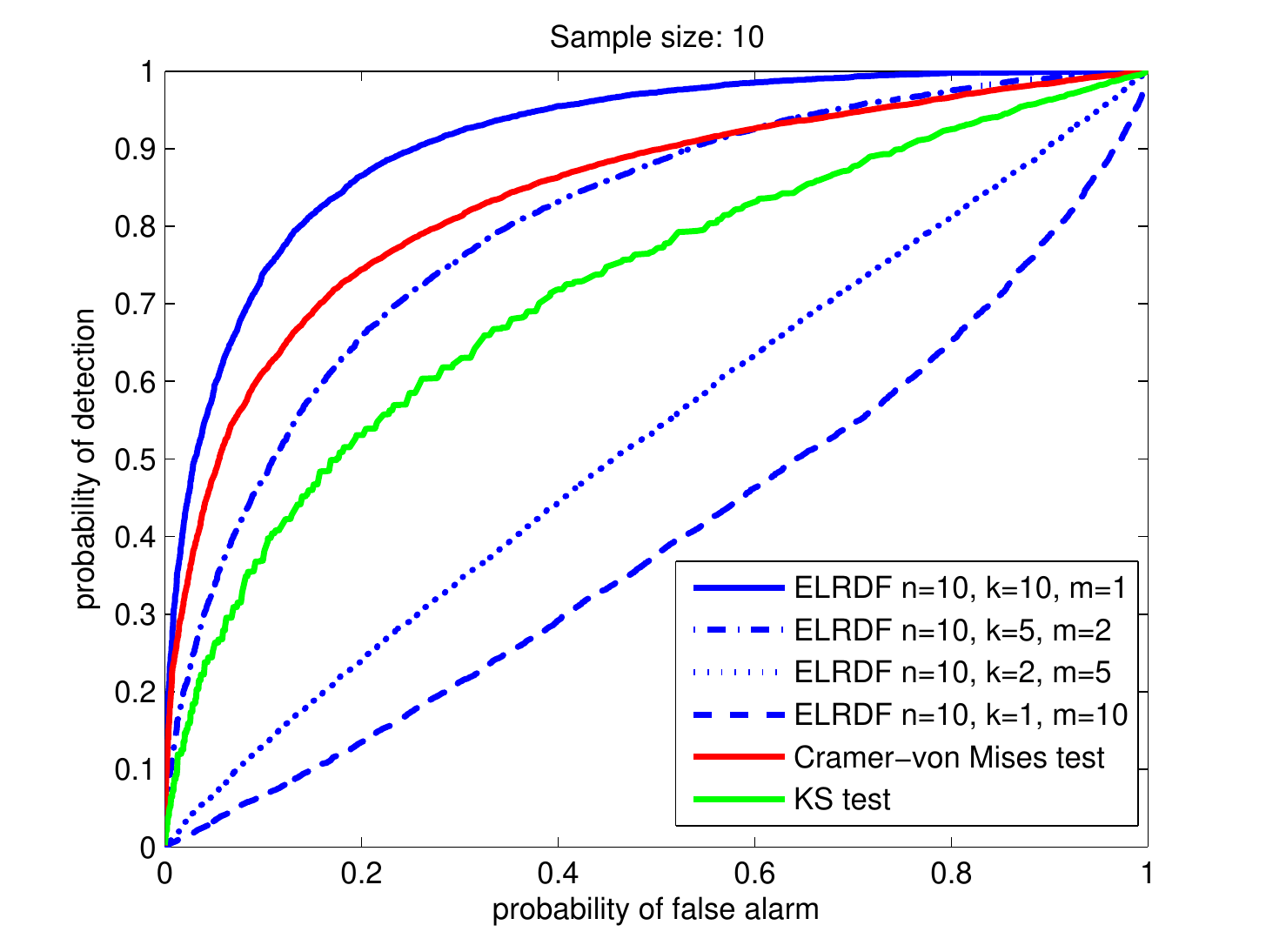}}
  \subfloat[$h_i=-3$]{\label{fig:eg3n}\includegraphics[width=3.5in]{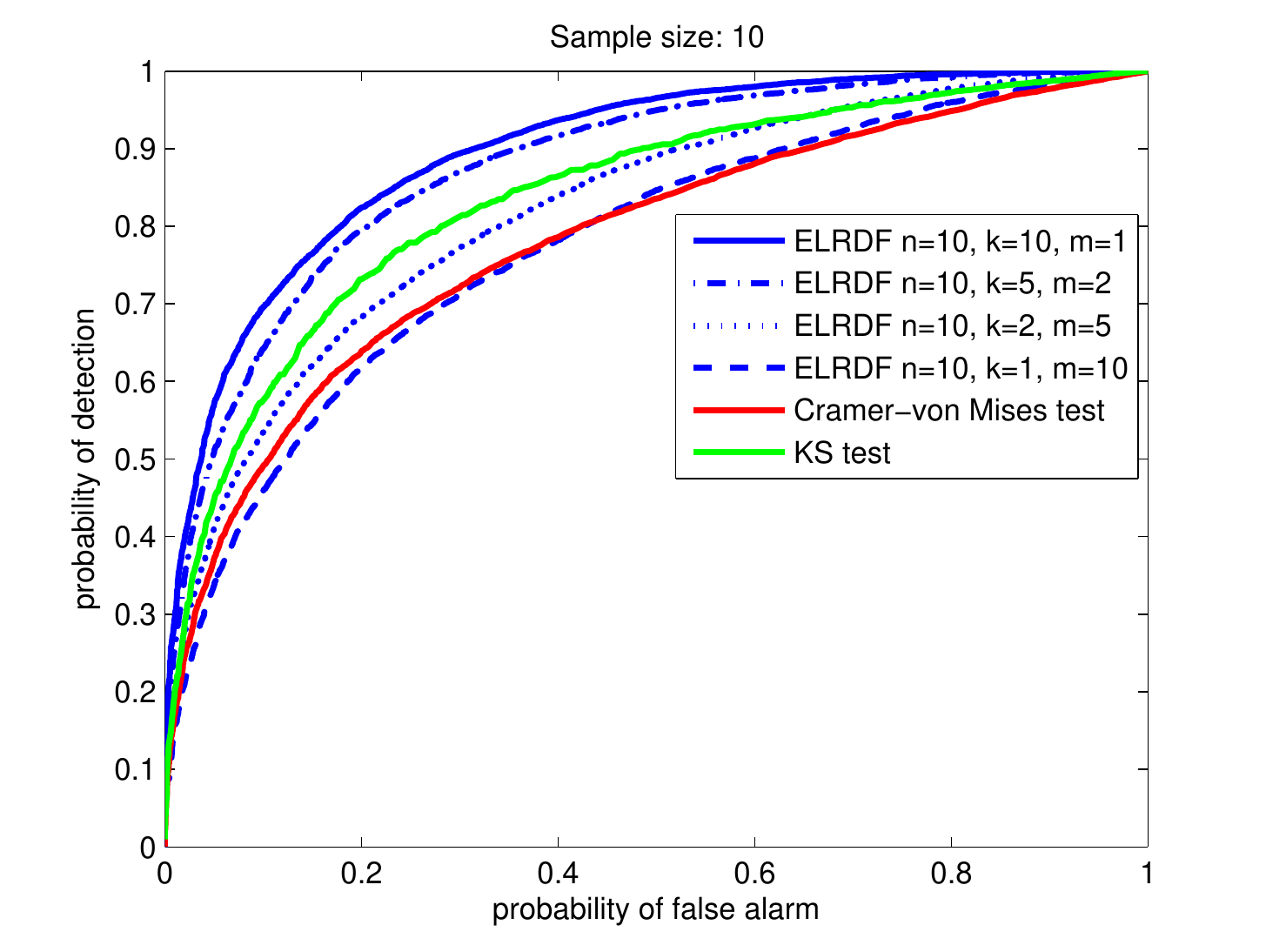}}
  \caption{ROC curves in slow fading scenario with degenerate distribution function constraint.}
  \label{fig:degperformance}
\end{figure*}
\begin{figure}[!t]
  \centering
  \includegraphics[width=3.5in]{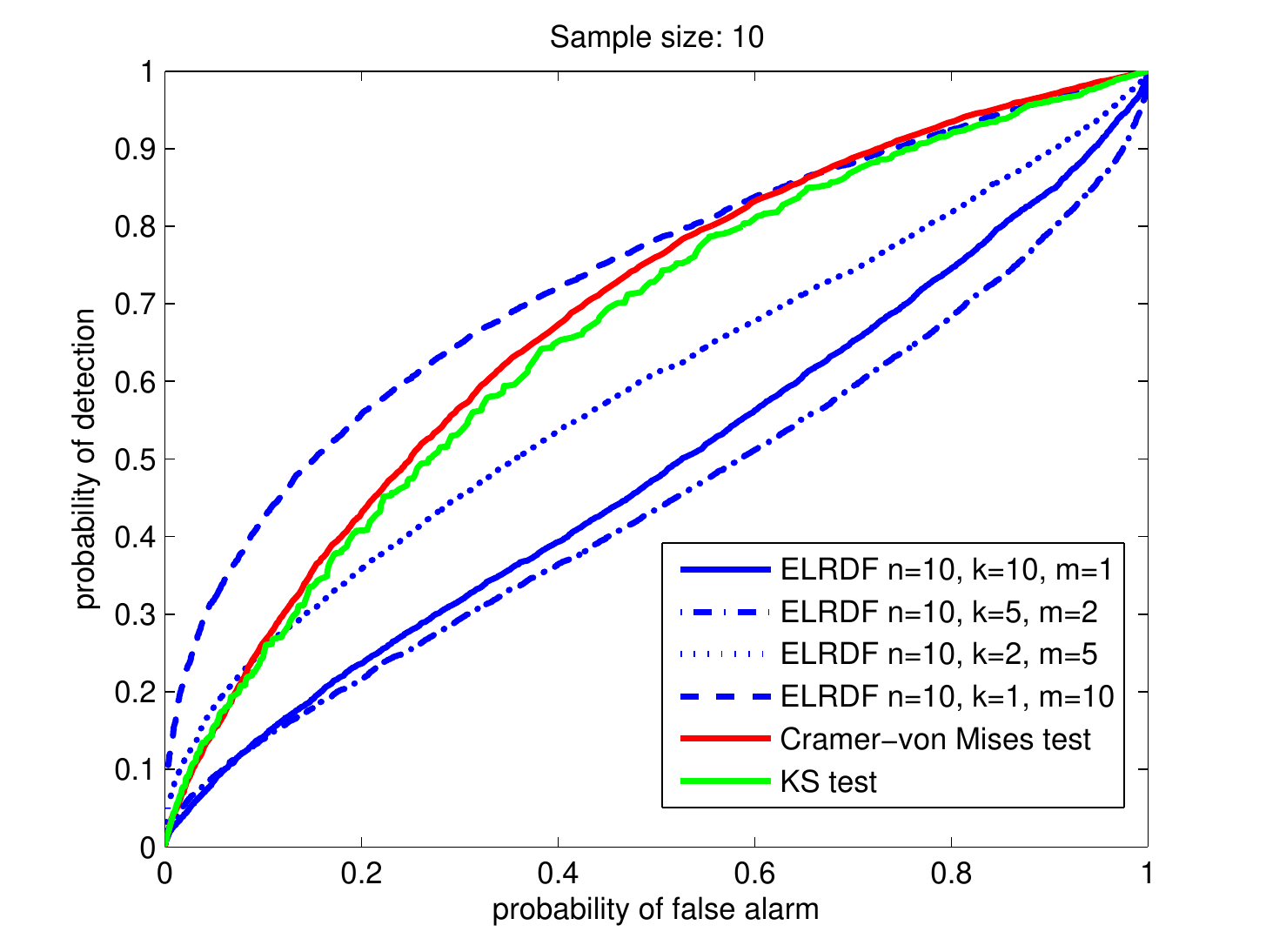}
  \caption{ROC curves in fast fading scenario with degenerate distribution function constraint.}
  \label{fig:eg4}
\end{figure}
\qed
\end{example}

\par
It is also of interest to examine the impact of the noise model on the test performances. In the two examples above, the only difference is the way the noise distribution is modeled. When there is no uncertainty and the noise distribution is stationary, the test is more accurate when more samples are used to evaluate the noise distribution. But in our case, it might be preferable to include a level of uncertainty to make the test robust. To illustrate this point, we compare the performance of the tests with the rich and degenerate distribution function constraint. Figure \ref{fig:compareslow} shows that in the slow fading scenario, the tests with degenerate constraint performs slightly better than those with rich distribution constraint when $h_i=3$. They perform quite similarly when $h_i=-3$. The ELRDF with degenerate constraint is applied with averaging over 10 groups. In the fast fading scenario, tests with a rich distribution constraint significantly outperform those with degenerate constraint where the ELRDF is applied without averaging. This result implies that uncertainty is necessary when the noise distribution is non-stationary.
\begin{figure}[!t]
  \centering
  \includegraphics[width=3.5in]{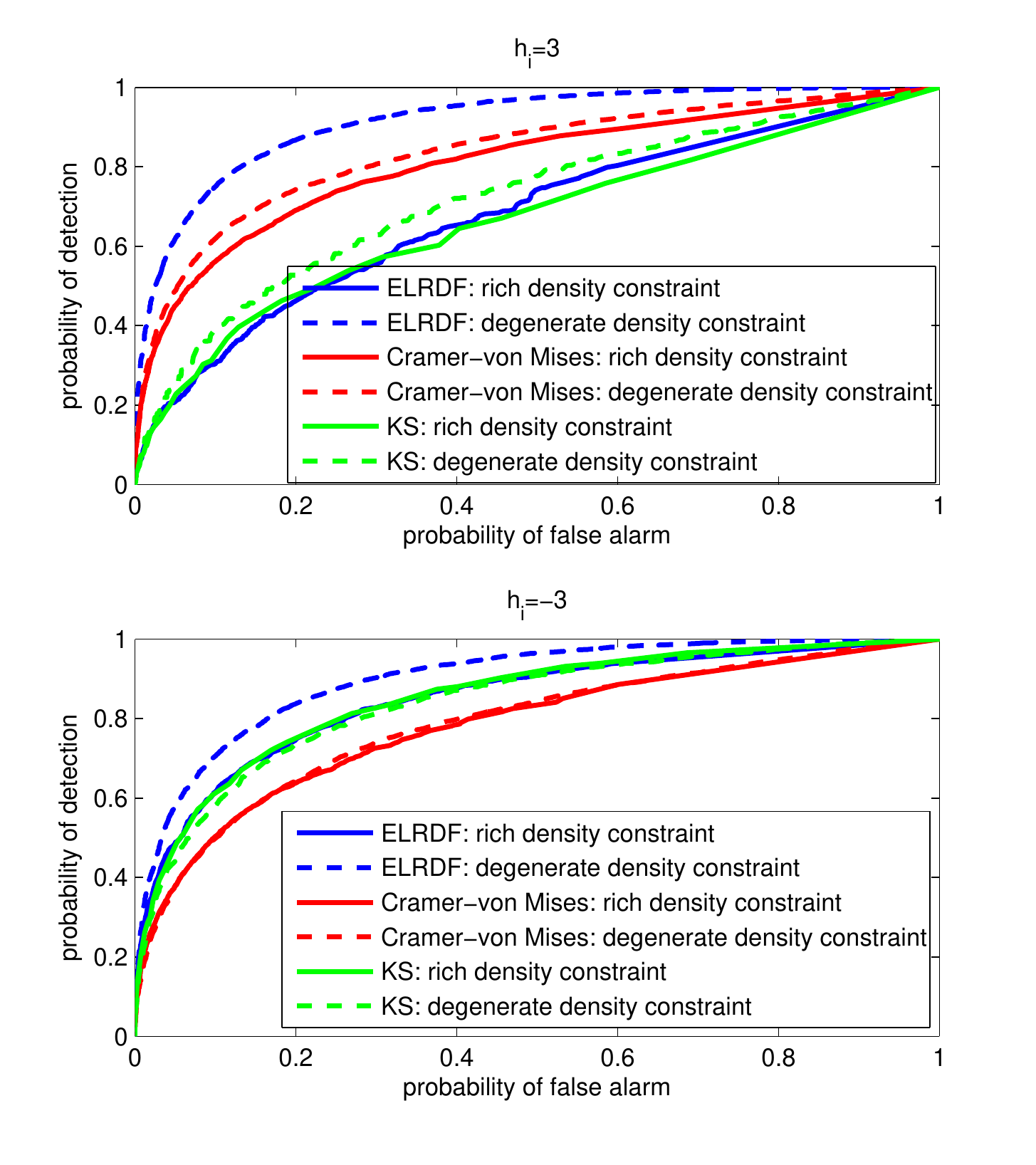}
  \caption{Comparison between rich and degenerate distribution function constraint in slow fading scenario.}
  \label{fig:compareslow}
\end{figure}
\begin{figure}[!t]
  \centering
  \includegraphics[width=3.5in]{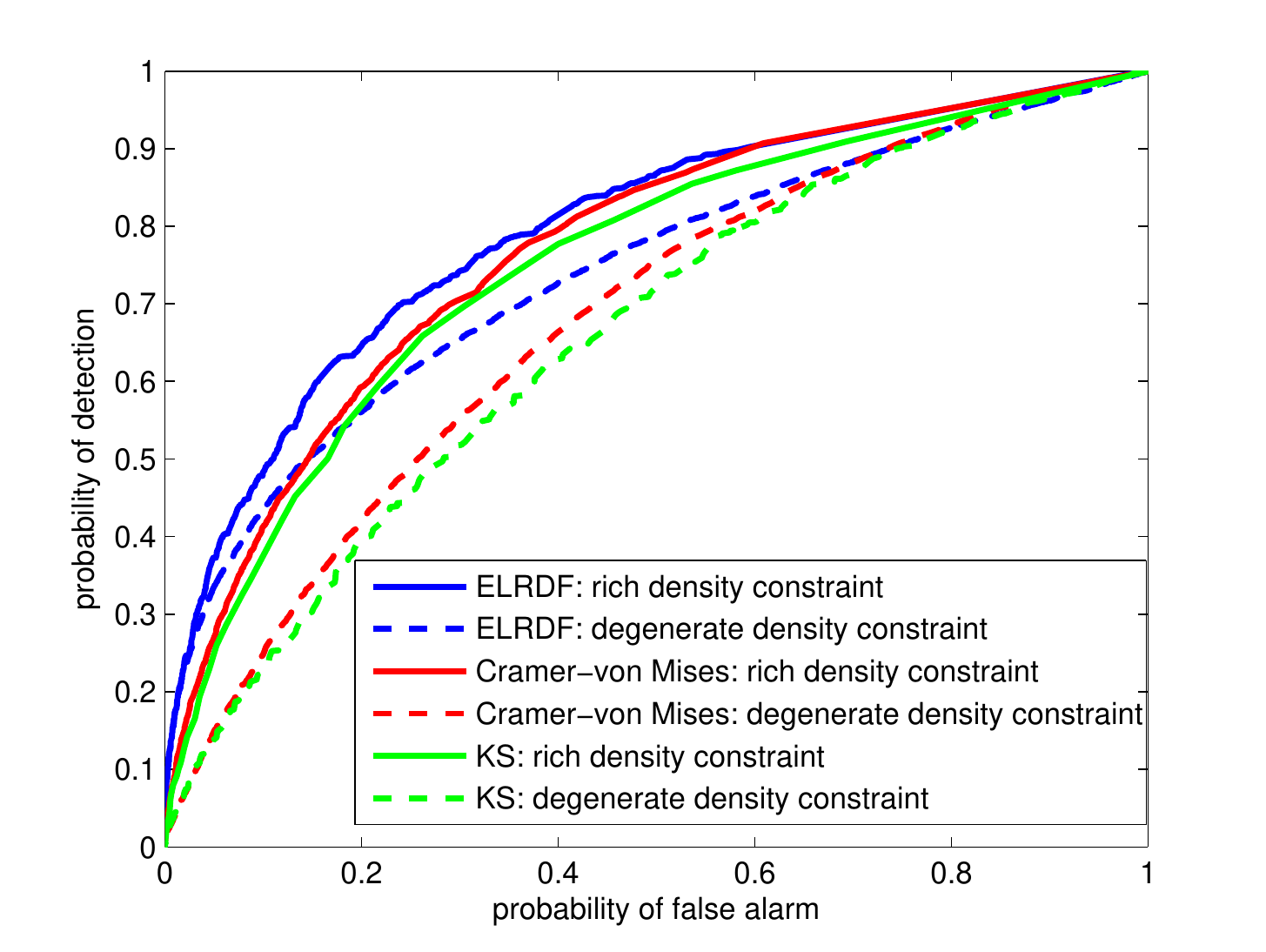}
  \caption{Comparison between rich and degenerate distribution function constraint in fast fading scenario.}
  \label{fig:comparefast}
\end{figure}

\section{Conclusion}\label{sec:6}
This work proposes a novel empirical likelihood ratio test with distribution function constraints (ELRDF), which is applicable to many applications in robust non-parametric detection problems. By providing a description of uncertainty using distribution function constraints, this test delivers better performance compared to many popular goodness-of-fit tests, such as the Kolmogorov-Smirnov test and the Cram\'{e}r-von Mises test. Also, the asymptotic optimality of the test is established. When the distribution function constrain is degenerate, the corresponding ELRDF is also devised. Several examples in communication systems are provided to show the performance of ELRDF with rich and degenerate distribution function constraints. With rich distribution function constraint, when the channel gain is a positive constant, the ELRDF is less powerful than the Cram\'{e}r-von Mises test. When it is a negative constant, ELRDF performs better. When the channel gain is random over a symmetric interval with respect to 0, ELRDF has better performance. With degenerate distribution function constraint, the ELRDF with suitable sample grouping outperforms both the Cram\'{e}r-von Mises test and the Kolmogorov-Smirnov test. We also show that robustness is necessary with our noise samples using results from the two examples.

\appendices
\section{Proof of Theorem \ref{thm1}}\label{app1}
\begin{proof}
1) By Sanov's theorem:
\begin{align}
    e_F(\Lambda)&=\lim\inf\limits_{n\rightarrow\infty}\inf\limits_{F\in\mathcal{F}}-\frac{1}{n}\log F(F_e\in\Omega_1^{\delta})\nonumber\\
    &\geq\inf\limits_{F\in\mathcal{F}}\inf\limits_{F_e\in\bar{\Lambda}_1}D(F_e||F)\geq\eta.\nonumber
\end{align}
The second inequality is from the fact that $\bar{\Lambda}_1=\{F_e:-\frac{1}{n}\log R(F^{\ast}, F_e)\geq\eta\}$ implied by the lower semicontinuity of $D(\cdot||F)$.

\par
2) This part of proof follows the technique used by Zeitouni and Gutman in \cite{zeitouni1991universal}. We first show that there exists some $n(\delta)$ such that $\Lambda_0\subseteq\Omega_0(n)$ for $n>n(\delta)$. Suppose that it is not true. Then there exists a sequence $n_k$, $k=1,2,3,\ldots$ such that $\mu_{n_k}\in\Lambda_0$ and $\mu_{n_k}\in\Omega_1(n_k)$. Since $\Lambda_0$ is compact, there exists a $\mu\in\Lambda_0$ such that $\mu_{n_k}\rightarrow\mu$. By the definition of $\Omega_1^{\delta}(n)$, $B(\mu_{n_k}, \delta)\subset\Omega_1^{\delta}(n_k)$. Then $B(\mu, \frac{1}{2}\delta)\subset\Omega_1^{\delta}(n_k)$ for infinitely many $n_k$. Then,
\begin{align}
    &\lim\inf\limits_{n\rightarrow\infty}\inf\limits_{F\in\mathcal{F}}-\frac{1}{n}\log F(F_e\in\Omega_1^{\delta}(n))\nonumber\\
    &\leq\lim\inf\limits_{n\rightarrow\infty}\inf\limits_{F\in\mathcal{F}}-\frac{1}{n}\log F(F_e\in B(\mu, \frac{1}{2}\delta))\nonumber\\
    &\leq\inf\limits_{F\in\mathcal{F}}\inf\limits_{\tilde{\mu}\in B(\mu, \frac{1}{2}\delta)}D(\tilde{\mu}||F)\nonumber\\
    &\leq\inf\limits_{F\in\mathcal{F}}D(\mu||F)\nonumber\\
    &\leq\eta.\nonumber
\end{align}
This draws a contradiction with the definition of $\Omega$. The second inequality comes from Sanov's theorem. The last inequality comes from the fact that $\mu\in\Lambda_0$. Therefore $\Lambda_0\subseteq\Omega_0(n)$ for all $n>n(\delta)$ and some $n(\delta)\in\mathbb{N}$. It follows that
\begin{align}
    &\lim\inf\limits_{n\rightarrow\infty}\inf\limits_{F\notin\mathcal{F}}-\frac{1}{n}\log F(F_e\in\Omega_0(n))\nonumber\\
    &\leq\lim\inf\limits_{n\rightarrow\infty}\inf\limits_{F\notin\mathcal{F}}-\frac{1}{n}\log F(F_e\in\Lambda_0).\nonumber
\end{align}
Hence prove the theorem.
\end{proof}

\section{Proof of Lemma \ref{lm1}}\label{app2}
\begin{proof}
One can write down the distribution density function $f_i$ for $i\geq 2$ as:
\begin{align}
    f_i(w_i)=\frac{n!}{(n-2)!}\binom{n-2}{i-2}\int_0^{1-w_i}x^{i-2}(1-w_i-x)^{n-i}dx.\nonumber
\end{align}
This expression can be interpreted as follows. There are $\frac{n!}{(n-2)!}$ permutations to pick 2 out of $n$ random variables. Then there are $\binom{n-2}{i-2}$ ways to pick $i-2$ random variables from the left $n-2$ random variables. The integral is the probability distribution evaluated at $w_i$. We further evaluate the integral.
\begin{align}
    &\int_0^{1-w_i}x^{i-2}(1-w_i-x)^{n-i}dx\nonumber\\
    =&\frac{1}{i-1}(1-w_i-x)^{n-i}x^{i-1}\Big |_0^{1-w_i}\nonumber\\
    &+\int_0^{1-w_i}\frac{n-i}{i-1}x^{i-1}(1-w_i-x)^{n-i-1}dx\nonumber\\
    =&\frac{n-i}{i-1}\int_0^{1-w_i}x^{i-1}(1-w_i-x)^{n-i-1}dx.\nonumber
\end{align}
Denote $A_i=\int_0^{1-w_i}x^{i-2}(1-w_i-x)^{n-i}dx$, then
\begin{align}
    A_{i+1}=\frac{i-1}{n-i}A_i.\nonumber
\end{align}
$A_2$ is evaluated to be
\begin{align}
    A_2&=\int_0^{1-w_i}x^{i-2}(1-w_i-x)^{n-i}dx\nonumber\\
    &=\frac{1}{n-1}(1-w_i)^{n-1}.\nonumber
\end{align}
Using $A_2$ and the aforementioned relationship, one can obtain the expression for $A_i$
\begin{align}
    A_i=\frac{(i-2)!(n-i)!}{(n-1)!}(1-w_i)^{n-1}.\nonumber
\end{align}
Substitute it in the original expression yields
\begin{align}
    f_i(w_i)=n(1-w_i)^{n-1}.\nonumber
\end{align}
One can easily obtain the same result for $i=1$. Also the mean of $W_i$ can be evaluated as $\mathbb{E}[W_i]=\frac{1}{n+1}$. Hence prove the lemma.
\end{proof}

\ifCLASSOPTIONcaptionsoff
  \newpage
\fi

\bibliographystyle{IEEEtran}
\bibliography{IEEEabrv,bibfile}

%

\begin{IEEEbiography}[{\includegraphics[width=1in,height=1.25in,clip,keepaspectratio]{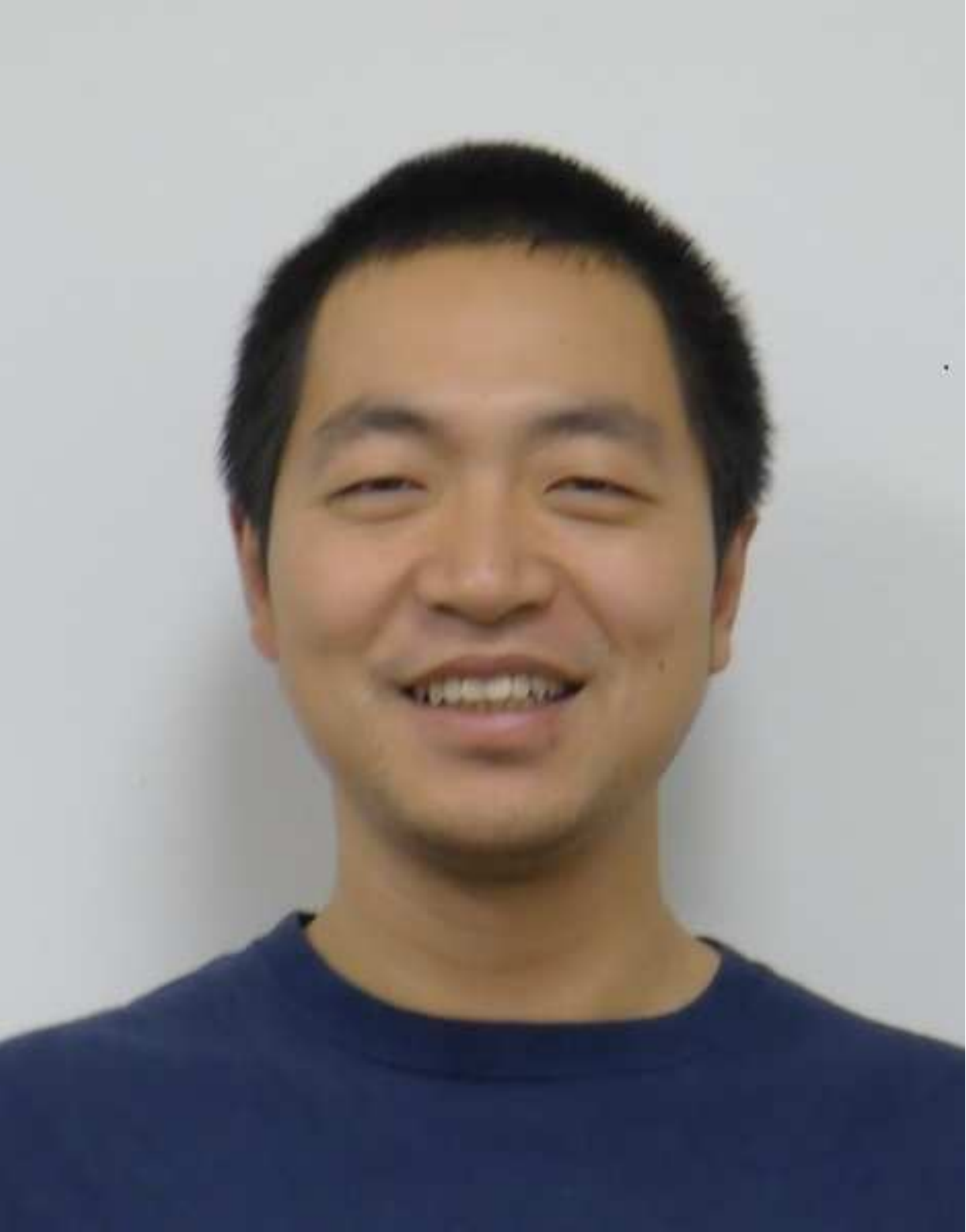}}]{Yingxi Liu}
received the B.E., M.E. degrees in electrical engineering from the Beijing University of Posts and Telecommunications (BUPT), Beijing, China, in 2005 and 2008, and the Ph.D. degree in electrical and computer engineering from the University of Texas at Austin in 2013. He worked as summer intern in Commscope, in Richardson, TX, 2011 and in Broadcom, in Sunnyvale, CA, 2012.

His research interests lie in several theoretical and practical aspects of signal processing and communication theory, including robust detection, non-parametric tests, stochastic modeling, as well as cognitive radio design and implementation.
\end{IEEEbiography}

\begin{IEEEbiography}[{\includegraphics[width=1in,height=1.25in,clip,keepaspectratio]{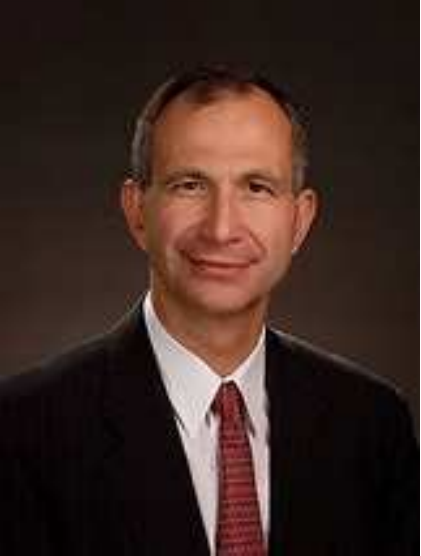}}]{Ahmed Tewfik}
received the B.Sc. degree from Cairo University, Cairo, Egypt, in 1982 and the M.Sc., E.E., and Sc.D. degrees from the Massachusetts Institute of Technology (MIT), Cambridge, MA, USA, in 1984, 1985, and 1987, respectively.

He is the Cockrell Family Regents Chair in Engineering and the Chairman of the Department of Electrical and Computer Engineering at the University of Texas, Austin, TX, USA. He was the E. F. Johnson Professor of Electronic Communications with the Department of Electrical Engineering at the University of Minnesota, Twin Cities, MN, USA, until September 2010. He worked at Alphatech, Inc. and served as a consultant to several companies. From August 1997 to August 2001, he was the President and CEO of Cognicity, Inc., an entertainment marketing software tools publisher that he co-founded, on partial leave of absence from the University of Minnesota. His current areas of research interests include medical imaging for minimally invasive surgery, programmable wireless networks, genomics and proteomics, neural prosthetics and audio signal separation. He has made seminal contributions in the past to food inspection, watermarking, multimedia signal processing and content based retrieval, wavelet signal processing and fractals.

Prof. Tewfik was a Distinguished Lecturer of the IEEE Signal Processing Society from 1997 to 1999. He received the IEEE Third Millennium Award in 2000. He was elected to the position of VP Technical Directions of the IEEE Signal Processing Society (SPS) in 2009 and served on the Board of Governors of the SPS from 2006 to 2008. He has given several plenary and keynote lectures at IEEE conferences.
\end{IEEEbiography}







\end{document}